\documentclass[a4paper]{scrartcl} 
\usepackage[utf8]{inputenc}
\usepackage[T1]{fontenc}
\usepackage{lmodern} 
\usepackage{amssymb, amsmath, amsfonts, amsthm} 
\usepackage{mathtools} \mathtoolsset{showonlyrefs} 
\usepackage{mathrsfs} 
\usepackage{enumerate}
\usepackage{color}
\usepackage{graphicx} 
\usepackage{mdwlist} 
\usepackage{paralist} 
\usepackage{cite}
\usepackage{ulem} 
\usepackage{microtype} 
\usepackage{hyperref} \usepackage[all]{hypcap} 

	\hypersetup{colorlinks=false} 

	\let\oldleft\left
	\let\oldright\right
	\renewcommand{\left}{\mathopen{}\mathclose\bgroup\oldleft}
	\renewcommand{\right}{\aftergroup\egroup\oldright}

	\theoremstyle{definition}
	\newtheorem{Definition}[equation]{Definition}
	\newtheorem{Remark}[equation]{Remark}
	
	\newtheorem{Claim}[equation]{Claim}

	\theoremstyle{plain}
	\newtheorem{Theorem}[equation]{Theorem}
	\newtheorem{Lemma}[equation]{Lemma}
	
	\newtheorem{Corollary}[equation]{Corollary}

	\let\op\operatorname

	\renewcommand{\epsilon}{\varepsilon}
	
	\renewcommand{\hat}{\widehat}
	\newcommand{\moo}[1]{ {\widetilde #1 } }

	\newcommand{\Z}{{\mathbb Z}}
	
	\newcommand{\R}{{\mathbb R}}

	\newcommand{\BigO}{{O}}
	\newcommand{\GL}{\operatorname{GL}}
	\newcommand{\SL}{\operatorname{SL}}
	
	\renewcommand{\O}{\operatorname{O}}

	\newcommand{\Eloc}[1]{\mathbb E_{0}\left[#1\right]} 

	\newcommand{\mat}[1]{\begin{pmatrix}#1\end{pmatrix}}

	\newcommand{\e}{\varepsilon}

	\DeclareMathOperator{\vol}{vol}

	\newcommand{\defeq}{\mathrel{\mathop:}=} 
	\newcommand{\eqdef}{=\mathrel{\mathop:}} 
	\newcommand{\inv}{^{-1}}
	\newcommand{\T}{^{\mathsf T}} 

	\newcommand{\supp}{\operatorname{supp}}

	\newcommand{\abs}[1]{\left|#1\right|}
	\newcommand{\norm}[1]{\left\|#1\right\|}

	\newcommand{\dfn}[1]{\textbf{#1}}

	\newcommand{\dd}[1]{\mathop{d#1}}
	
	\newcommand{\ddx}[1]{\frac{\dd{}}{\dd #1}}
	\newcommand{\pypx}[2]{\frac{\partial #1}{\partial #2}}
	\newcommand{\ppx}[1]{\frac{\partial}{\partial #1}}

	\newcommand{\dmu}[2]{\mathop{d#1}(#2)}

	\newcommand{\p}[1]{\section{ #1 } }

\begin{document}
	\title{The number of points from a random lattice that lie inside a ball}
	\author{Samuel Holmin \thanks{The author was partially supported by the Swedish Research Council.} }
	
	\date{\today}

	\maketitle

	\begin{abstract}
		We prove a sharp bound for the remainder term of the number of lattice points inside a ball, when averaging over a compact set of (not necessarily unimodular) lattices, in dimensions two and three. We also prove that such a bound cannot hold if one averages over the space of all lattices.	
	\end{abstract}


\let\oldll\ll \renewcommand{\ll}{\oldll} 
\renewcommand{\lll}{\lessapprox} 
\newcommand{\ESL}[1]{\mathbb E_{1}\left[#1\right]} 

\p {Introduction}
Let $\Omega$ be the (closed) standard unit ball in $\R^n$. A \dfn{lattice} in $\R^n$ is a set of the form $X\cdot\Z^n\subseteq \R^n$ for some $X\in\GL_n(\R)$. The set of all lattices may be identified with the space $\GL_n(\R)/\GL_n(\Z)$, and we equip it with a measure $\mu$ induced by the Haar measure on $\GL_n(\R)$. Let $N_X(t)$ be the number of points from the lattice $X\Z^n$ inside the ball $t\Omega$ of radius $t$. We have $N_X(t)=\#(X\Z^n\cap t\Omega)=\#(\Z^n\cap t\Omega_X)$, where $\Omega_X\defeq X\inv \Omega$.
Let $E_X(t)\defeq N_X(t)-\vol(t\Omega_X)$.
Consider the set of unit cubes centered at the set of integer points $u\in\Z^n$.
Since $N_X(t)$ equals the number of cubes whose center is inside $t\Omega_X$, which coincides with the volume of the union of these cubes, we can write
\begin{gather}
	N_X(t)=\vol(t\Omega_X)+\sum_{\text{cubes $T$ intersecting }\partial(t\Omega)}Y_T,
\end{gather}
where $Y_T$ equals $\vol(T\setminus t\Omega_X)$ if the center of $T$ is inside $t\Omega$, and $Y_T$ equals $-\vol(T\cap t\Omega_X)$ otherwise. There are approximately $\vol(\partial(t\Omega_X))=t^{n-1}\vol(\partial(\Omega_X))$ correction terms $Y_T$, each bounded, so it follows that $N_X(t)$ is asymptotic to $t^n\vol(\Omega_X)$. Heuristically, if the correction terms $Y_T$ were i.i.d.\ random variables, the central limit theorem would imply that the standard deviation of the remainder term $E_X(t)=\sum_T Y_T$ is approximately proportional to $\sqrt{\vol(\partial(t\Omega_X))}$ for large $t$. This suggests that $\abs{E_X(t)}$ should be of the order $t^{(n-1)/2}$ for fixed $X$.

Let $\delta>0$ be a small arbitrary constant. For the integer lattice $\Z^2$, Hardy conjectured that $\abs{E_{\Z^2}(t)}=\BigO( \sqrt{\vol(\partial(t\Omega))} \cdot t^\delta )=\BigO(t^{1/2+\delta})$ as $t\to\infty$ \cite{hardycareful}.
It is known that $\abs{E_X(t)} \neq \BigO(t^{1/2})$ for every lattice in $\R^2$, due to Nowak \cite{nowak126}, and the best known upper bound is $\abs{E_{X}(t)}=\BigO(t^{131/208+\delta}),$ where $131/208 \approx 0.62981$, due to Huxley \cite{huxley3}. Hardy's conjecture holds on average in the sense that $\sqrt{\frac1t\int_0^{t}|E_{X}(\tau)|^2\dd \tau}=\Theta(t^{1/2})$, due to Bleher \cite{bleher}.


In three dimensions, it is known that $\abs{E_X(t)}\neq O(t)$, due to Nowak \cite{nowak129}, and the best known upper bound for arbitrary lattices in $\R^3$ is $\abs{E_{X}(t)}= O(t^{63/43+\delta})$, where $63/43\approx 1.465$, due to Müller \cite{muller122}, with the improvement $\abs{E_{\Z^3}(t)}= O(t^{21/16+\delta})$ for the integer lattice $\Z^3$, where $21/16=1.3125$, due to Heath-Brown \cite{sphereproblem}. On average, we have $\sqrt{\frac1t\int_t^{2t}|E_{X}(\tau)|^2\dd \tau}=\BigO(t^{1+\delta})$, see \cite{iosevich}.


The main result of this paper is that the bound $\BigO(t^{(n-1)/2+\delta})$ holds on average in dimensions two and three, when averaging over any compact set of lattices:

\begin{Theorem} \label{maintheorem} \label{bam}
	Let $n=2$ or $n=3$. Fix a compact subset $L_0$ of $\GL_n(\R)/\GL_n(\Z)$. Then there exists an integer $m>0$ such that
	\begin{gather}
		\sqrt{\Eloc{\abs{E_X(t)}^2}} = \BigO(t^{(n-1)/2}(\log t)^m)
	\end{gather}
	as $t\to\infty$, where $\Eloc{f(X)}\defeq\int_{L_0} f(X)\dmu\mu X$ is the mean of $f$ over $L_0$. 
\end{Theorem}

This bound is sharp in the sense that $\abs{E_X(t)}\neq o(t^{(n-1)/2})$ for every lattice in $n\geq 3$ dimensions (this result is due to Landau \cite{landau}).
It is not known for any $n\geq 2$ if there exists for each $\delta>0$ some $X$ such that $\abs{E_X(t)}=\BigO(t^{(n-1)/2+\delta})$, but Schmidt proved in \cite{schmidt} that $\abs{E_X(t)} = \BigO(t^{n/2+\delta})$ for almost every lattice, when $n\geq 2$. The best general bound for $n\geq 5$ is $\abs{E_X(t)}=\BigO(t^{n-2})$, due to Götze \cite{goatse}, and this bound is attained by the integer lattices (to be specific, $\abs{E_{\Z^n}(t)}\neq o(t^{n-2})$ for every $n\geq 4$, see Krätzel \cite{kratzel}). See \cite{survey} for an excellent survey on results about lattice points in convex domains.

The assumption in Theorem \ref{bam} that $L_0$ is compact cannot be removed when $n=3$: as Corollary \ref{altcorollary} below shows, if we average over the set $L_{a,b}=\{X\in \GL_3(\R)/\GL_3(\Z): 0<a\leq \abs{\det X}\leq b<\infty \}$, which is not compact, then we get both a lower and an upper bound with an exponent strictly larger than what Theorem \ref{bam} guarantees. The failure of the heuristic in this case may be explained by the fact that $L_{a,b}$ contains lattices with arbitrarily short lattice vectors.

\begin{Theorem} \label{alttheorem}
	For any fixed $n\geq 3$, we have
	\begin{gather}
		\sqrt{\ESL{\abs{E_X(t)}^2}} = \Theta( \sqrt{\vol(t\Omega)} ) = \Theta(t^{n/2})
	\end{gather}
	as $t\to\infty$, where $\ESL{f(X)}\defeq \int_{\SL_n(\R)/\SL_n(\Z)}f(X)\dmu{\mu_1} X$ is the mean of $f$ over the set of all lattices in $\SL_n(\R)/\SL_n(\Z)$, and where $\mu_1$ is the normalized Haar measure on $\SL_n(\R)$.
\end{Theorem}

\newcommand{\Eab}[1]{\mathbb E_{a,b}\left[#1\right]}

\begin{Corollary} \label{altcorollary}
	Fix $0<a<b$. For any fixed $n\geq 3$, we have
	\begin{gather}
		\sqrt{\Eab{\abs{E_X(t)}^2}} = \Theta(t^{n/2})
	\end{gather}
	as $t\to\infty$, where $\Eab{f(X)}\defeq\int_{L_{a,b}} f(X)\dmu\mu X$ is the mean of $f$ over $L_{a,b}=\{X\in \GL_n(\R)/\GL_n(\Z): a\leq \abs{\det X}\leq b\}$.\footnote{Note that averaging over the whole set $\GL_n(\R)/\GL_n(\Z)$ does not make sense, since $\GL_n(\R)/\GL_n(\Z)$ has infinite covolume and consequently the expected value of any constant would be infinite.}
\end{Corollary}

The corresponding statement of Theorem \ref{maintheorem} for orthogonal lattices (that is, lattices $X\Z^n$ where $X$ is a diagonal matrix) was proved by Hofmann, Iosevich, Weidinger in \cite{ellipsoids}, and our proof of Theorem \ref{bam} is inspired by theirs.

This paper is organized as follows. Sections \ref{decompsection} through section \ref{sectionofmainproof} are dedicated to the proof of Theorem \ref{bam} for $n=3$. We sketch in section \ref{sectionmain2} how the given proof may be modified for the slightly easier case $n=2$. Theorem \ref{alttheorem} is an easy consequence of the mean value formulas of Siegel and Rogers; we prove Theorem \ref{alttheorem} and Corollary \ref{altcorollary} in section \ref{proofoversl}.

\begin{Remark}
	The actual measure used in Theorem \ref{maintheorem} is not important; the proof holds for any measure of the form $f(X)\dd X$ and any compact set $L_0$ of $\GL_n(\R)$, where $\dd X$ is the Euclidean measure on the entries of the matrix $X$ and $f:\GL_n(\R)\to \R^+$ is a function which is bounded above and below in $\R^+=\{x\in\R: x>0\}$ throughout $L_0$.

	For instance, one may use the following natural measure for generating random lattices close to a given lattice.
	Fix a matrix $X_0\in\GL_n(\R)$.
	We generate random vectors $x_1,\ldots,x_n$, where each vector $x_i$ is generated by a uniform probability measure on vectors sufficiently close to the $i$th column of $X_0$, and then we let $x_1,\ldots,x_n$ be the basis vectors of our random lattice. This corresponds to taking $f(X)=1$ for all $X$ and taking $L_0\defeq \{X_0+tE: \abs t\leq \e\}$, where $E$ is the $n\times n$-matrix of all ones, and $\e>0$ is sufficiently small such that $L_0$ does not contain any singular matrices. 
\end{Remark}


\newcommand{\BOUND}{{\mathcal U(t)}}
\newcommand{\somevectors}{ \Z^3(\BOUND) } 

\p {Notation}
Throughout this paper, we will assume that the parameter $t>1$ is large.
We will write $f(t)\lll g(t)$ if there exists a constant $c>0$ and an integer $m\geq 0$ such that $\abs{f(t)}\leq \abs{c g(t) (\log t)^m}$ for all sufficiently large $t$. We see that $\lll$ is a transitive relation. As customary, we will write $f(t)\ll g(t)$ if there exists a constant $c$ such that $\abs{f(t)}\leq \abs{c g(t)}$ for all sufficiently large $t$.

Given a function $f:\R^k\to\R$ for some $k$, we write $\hat f(\xi)=\int_{\R^k}f(x)e^{-2\pi i x\cdot \xi}\dd x$ for its Fourier transform.

We will write $\Z^n(a)$ for the set of all nonzero integer vectors $k=(k_1,\ldots,k_n)$ such that $\abs{k_i}\leq a$ for each $1\leq i\leq n$. For a vector $k$ and a matrix $X$, we will write $\norm k_X\defeq \norm{(X\inv)\T k}$. Finally, we will frequently use the notation $\moo k\defeq (N\inv)\T k$ where $N$ is a given upper triangular matrix which will be clear from context.


\p {Decomposition of the Haar measure} \label{decompsection}
Let $\mu$ be the Haar measure on $\GL_3(\R)$. The measure $\mu$ induces a measure on the quotient space $\GL_3(\R)/\GL_3(\Z)$, and we will abuse notation by denoting both of these measures by the symbol $\mu$. Let $\mathcal F\subseteq \GL_3(\R)$ be a fundamental domain relative to $\GL_3(\Z)$. If $f:\GL_3(\R)/\GL_3(\Z)\to\R$ is an integrable function, we shall write $f(X)\defeq f(X\cdot\GL_3(\Z))$ for $X\in\GL_3(\R)$, and then
\[\int_{\GL_3(\R)/\GL_3(\Z)}f(X)\dmu\mu X=\int_{\mathcal F\subseteq \GL_3(\R)}f(X)\dmu\mu X,\]
where in the right-hand side we are integrating with respect to the measure on $\GL_3(\R)$.

We will use the Iwasawa decomposition $\GL_3(\R)=\mathcal{K\cdot A\cdot N}$ where $\mathcal K=\O_3(\R)$ is the group of orthogonal matrices, $\mathcal A$ is the group of diagonal matrices with positive diagonal entries, and $\mathcal N$ is the group of upper triangular matrices with ones on the diagonal. If $X\in \GL_3(\R)$, then there is a unique $(K,A,N)\in\mathcal{K\times A\times N}$ such that $X=KAN$. Let $\mathcal N^+$ be the set of all matrices $N\in \mathcal N$ such that all entries of $N$  above the diagonal belong to the interval $[1,2)$. (We will later use the fact that the entries of $N\in\mathcal N^+$ are not close to zero.) By performing Euclid's algorithm on the columns of $N$ using elementary column operations, one can show that there exists for any $X=KAN$ some matrix $U\in\GL_3(\Z)$ such that $XU\in\mathcal K\cdot\mathcal A\cdot\mathcal N^+$, which shows that the set 
$\mathcal{K\cdot A\cdot N^+}\subseteq \GL_3(\R)$
contains a fundamental domain $\mathcal F^+$ relative to $\GL_3(\Z)$. 

The Haar measure $\mu$ on $\GL_3(\R)$ can be expressed in terms of the left-invariant Haar measures on $\mathcal K,\mathcal A$ and $\mathcal N$ as follows. Let $\mathcal{R\defeq A\cdot N}$ be the group of upper triangular matrices with positive diagonal elements. The Haar measure on $\mathcal A$ is $\dd A=\dd{b_1}\dd{b_2}\dd{b_3}/(b_1b_2b_3)$ where $b_1,b_2,b_3$ are the diagonal elements of $A\in\mathcal A$, and the Haar measure on $\mathcal N$ is $\dd N=\dd{\eta_1}\dd{\eta_2}\dd{\eta_3}$ where $\eta_1,\eta_2,\eta_3$ are the entries of $N\in \mathcal N$ above the diagonal. Write $\mu_{\mathcal K}$ for the (appropriately normalized) Haar measure on $\mathcal K$. Theorem 8.32 from \cite{knapp} implies that for any integrable function $f$, we have
\[\int_{\GL_3(\R)} f(X)\dmu{\mu}X=\int_{\mathcal N}\int_{\mathcal A}\int_{\mathcal K} f(KAN)\dfrac{\Delta_{\mathcal R}(AN)}{\Delta_{\GL_3(\R)}(AN)}\dfrac{\Delta_{\mathcal N}(N)}{\Delta_{\mathcal R}(N)}\dmu{\mu_{\mathcal K}} K\dd A\dd N\]
where $X=KR=KAN$, and $\Delta_G:G\to\R^+$ is the modular function associated with a topological group $G$. 
Let us write $\Delta(A,N)\defeq \frac{\Delta_{\mathcal R}(AN)}{\Delta_{\GL_3(\R)}(AN)}\frac{\Delta_{\mathcal N}(N)}{\Delta_{\mathcal R}(N)}$. The modular functions can be computed (in fact, one may show that $\Delta_{\GL_3(\R)}=\Delta_{\mathcal N}=1$, and $\Delta_{\mathcal R}(R)=b_1^2b_3^{-2}$ where $b_1,b_2,b_3$ are the diagonal elements of $R$), but all we will need is that $\Delta$ is bounded when restricted to a compact set, which follows from the fact that the modular functions are continuous and positive (see \cite{knapp}).

For our purposes, the parametrization
\begin{align}
	N&=\mat{ 1&\eta_1&\eta_2\\0&1&\eta_3\\0&0&1 }\in\mathcal N^+, &\eta_i\in[1,2), \label{param}\\
	A&=\mat{1/\sqrt{a_1}&0&0 \\ 0&1/\sqrt{a_2}&0 \\ 0&0&1/\sqrt{a_3}}\in\mathcal A, &a_i\in(0,\infty),
\end{align}
will be useful. (The forthcoming expression \eqref{fi} will take on a simpler form.) We get the Jacobian $\abs{\pypx{(b_1,b_2,b_3)}{(a_1,a_2,a_3)}  } = 2^{-3}(a_1a_2a_3)^{-2}$.   
Writing $\Delta(a,\eta)\defeq \Delta(A,N)$, and letting $f$ be a non-negative integrable function on $\GL_3(\R)/\GL_3(\Z)$, we obtain 
\begin{gather}
	\int\limits_{\GL_3(\R)/\GL_3(\Z)}f(X)\dmu\mu X = \int_{\mathcal F^+ }f(X)\dmu\mu X \leq \\
	\int_{\mathcal{K\cdot A\cdot N^+}}f(X)\dmu\mu X = 
	\iiint\limits_{\substack{K\in\mathcal K \\ a\in(0,\infty)^3\\\eta\in[1,2)^3 }} f(KAN) \dfrac{\Delta(a,\eta)}{2^3(a_1a_2a_3)^2} \dd a\dd \eta \dmu{\mu_{\mathcal K}} K,
\end{gather}
where $\dd a=\dd{a_1}\dd{a_2}\dd{a_3}$ and $\dd\eta=\dd{\eta_1}\dd{\eta_2}\dd{\eta_3}$ are the standard Lebesgue measures.

Integrating over the compact set $L_0\subseteq \GL_3(\R)/\GL_3(\Z)$ with respect to the measure $\mu$ corresponds to integrating over the compact set
\begin{gather}
	L_0'\defeq L_0\cdot\GL_3(\Z)\cap \mathcal F^+\subseteq \GL_3(\R) \label{lop}
\end{gather}
with respect to the measure $\dd a\dd \eta\dmu{\mu_{\mathcal K}}K$.
For each $i=1,2,3$, let $\psi_i$ be the characteristic function of the smallest closed interval contained in $(0,\infty)$ which contains all values that $a_i$ assumes when $X=KAN$ ranges over the compact set $L'_0$. Since $g(X)\defeq\abs{E_X(t)}^2$ is rotation invariant  (that is, $g(KX)=g(X)$ for all $K\in\mathcal K,X\in\GL_3(\R)$) and non-negative, we have
\begin{gather}
	\int_{L_0} \abs{E_X(t)}^2\dmu\mu X \leq \int_{[1,2)^3}\int_{(0,\infty)^3} \abs{E_{AN}(t)}^2 \dfrac{\Delta(a,\eta)}{2^3(a_1a_2a_3)^2} \psi_1(a_1)\psi_2(a_2)\psi_3(a_3)\dd a\dd \eta.
\end{gather}
The support of $\psi_1\psi_2\psi_3$ is contained in $(0,\infty)^3$, so for simplicity of notation, we will allow the inner integral to range over all of $\R^3$.
Since ${\Delta(a,\eta)}/({2^3(a_1a_2a_3)^2})$ and $4\pi \abs{\det A}^2$ are bounded above and below throughout the support of $\psi_1\psi_2\psi_3$, a bound of the right-hand side above will be equivalent, up to constants, to a bound of
\begin{gather}
	\int_{[1,2)^3}\int_{\R^3} \abs{E_{AN}(t)}^2 \dfrac{\Delta(a,\eta)}{2^3(a_1a_2a_3)^2} \dfrac{2^3(a_1a_2a_3)^2}{\Delta(a,\eta)}4\pi\abs{\det A}^2 \psi_1(a_1)\psi_2(a_2)\psi_3(a_3)\dd a\dd \eta \\
	= \int_{[1,2)^3}\int_{\R^3} \abs{E_{AN}(t)}^2 \psi(a)\dd a\dd \eta, \label{decomposition}
\end{gather}
where we have defined
\begin{gather}
	\psi(a)\defeq 4\pi \abs{\det A}^2 \psi_1(a_1)\psi_2(a_2)\psi_3(a_3). \label{def_of_psi}
\end{gather}
(It is convenient to introduce the factor $4\pi\abs{\det A}^2$ as it will later be cancelled by a factor appearing from $\abs{E_{AN}(t)}^2$.) Thus, in order to bound $\int_{L_0} \abs{E_X(t)}^2\dmu\mu X$, it suffices to bound \eqref{decomposition}.

\p {Setup}
We define a smoothed version of
\[N_X(t) = \sum_{k\in\Z^3}\chi_{t\Omega_X}(k)\]
by
\begin{gather}
	N_X^\e(t)\defeq \sum_{k\in\Z^3}\chi_{t\Omega_X}*\rho_\e(k) \label{strawberry}
\end{gather}
where $\rho:\R^3\to\R$ is a mollifier and $\rho_\e(x)\defeq \e^{-3}\rho(x/\e)$ for a parameter $\e=\e(t)>0$. (Recall that a mollifier is a smooth, non-negative function with compact support and unit mass.) We define $\rho(x)\defeq \rho_0(x_1)\rho_0(x_2)\rho_0(x_3)$ where $\rho_0:\R\to\R$ is an even mollifier such that $\abs{\hat{\rho_0}(y)} \ll e^{-\sqrt y}$ for large $y$; see \cite{ingham} for the construction of such a function $\rho_0$. We obtain the asymptotics
\begin{gather}
	\abs{\hat{\rho}(x)} \ll e^{-\sqrt{\abs{x_1}}-\sqrt{\abs{x_2}}-\sqrt{\abs{x_3}}} \ll e^{-\sqrt{\norm x }} \label{chocolate}
\end{gather}
as $\norm x\to\infty$, by the inequality $( \sqrt{\abs{x_1}}+\sqrt{\abs{x_2}}+\sqrt{\abs{x_3}} )^4 \geq x_1^2+x_2^2+x_3^2$. Note that the Fourier transform $\hat\rho$ is real-valued since $\rho$ is an even function.

Since the convolution $\chi_{t\Omega_X}*\rho_\e$ is smooth, we may apply the Poisson summation formula to the sum \eqref{strawberry}, and since both of the functions $\chi_{t\Omega_X}$ and $\rho_\e$ have compact support, the convolution theorem $\hat{\chi_{t\Omega_X}*\rho_\e}=\hat{\chi_{t\Omega_X}}\cdot \hat{\rho_\e}$ holds.
Moreover, $\hat{\chi_{t\Omega_X}}(0,0,0)=\int_{t\Omega_X}1=t^3\vol{\Omega_X}$ and $\hat{\rho_\e}(0,0,0)=\int\rho_\e=1$, so we get
\[N_X^\e(t)=t^3\vol\Omega_X+\sum_{k\neq(0,0,0)}\hat{\chi_{t\Omega_X}}(k)\hat{\rho_\e}(k)\eqdef t^3\vol\Omega_X+E_X^\e(t).\]

We first show that the function $N_X^\e$ approximates $N_X$ well:
\begin{Lemma} \label{Nlemma}
	There exists a constant $R>0$ such that
	\[N_X^\e(t-R\e)\leq N_X(t)\leq N_X^\e(t+R\e),\]
	where $R$ only depends on the mollifier $\rho$.
\end{Lemma}
\begin{proof}
Let $R$ be the radius of a ball centered at the origin which contains the support of $\rho$, so that the support of $\rho_\e$ is contained in a ball of radius $\e R$.
Consider
\[\chi_{t\Omega_X}*\rho_\e(k)=\int \rho_\e(x)\chi_{t\Omega_X}(k-x)\dd x.\]
The integral ranges over all $x\in\supp\rho_\e$, so we may assume that $\|x\|\leq \e R$ inside the integral. If $k$ is inside $t\Omega_X$ and at a distance at least $\e R$ from the boundary $\partial(t\Omega_X)$, then $\chi_{t\Omega_X}(k-x)=1$, so the integral becomes $\int \rho_\e(x)\dd x=1$, which agrees with $\chi_{t\Omega_X}(k)=1$. If on the other hand $k$ is outside $t\Omega_X$ and at a distance at least $\e R$ from the boundary $\partial(t\Omega_X)$, then $\chi_{t\Omega_X}(k-x)=0$, so the integral vanishes and again agrees with $\chi_{t\Omega_X}(k)=0$. Finally, if $k$ is at a distance at most $\e R$ from the boundary $\partial(t\Omega_X)$, then since $0\leq \chi_{t\Omega_X}\leq 1$ and $\rho_\e$ is nonnegative, the integral is bounded below by $0$ and above by $\int\rho_\e=1$. We have thus proved that $\chi_{t\Omega_X}*\rho_\e$ equals $\chi_{t\Omega_X}$ at all points at a distance at least $\e R$ from the boundary of $t\Omega_X$, and at all other points it assumes a value in $[0,1]$. This proves the lemma, since $N_X(t)$ counts the number of lattice points inside $t\Omega_X$, while $N_X^\e(t-R\e)$ counts each of these with a weight at most 1, and $N_X^\e(t+R\e)$ counts all the same lattice points, plus a few more with various weights in $[0,1]$.
\end{proof}


Using the lemma, we arrive at:
\begin{Claim} \label{vroom}
	To prove Theorem \ref{bam} for $n=3$ it suffices to prove that
	\begin{gather}
		\int_{[1,2)^3}\int_{\R^3} \abs{E_{AN}^\e(t)}^2 \psi(a)\dd a\dd \eta \lll t^2 \label{hypothesis}
	\end{gather}
	for all $\e=\e(t)$ such that $\e\geq 1/t$ for all sufficiently large $t$.
\end{Claim}
\begin{proof}
	Lemma \ref{Nlemma} implies that
	\begin{gather}
		E_X(t)\leq E_X^{\e_0}(t+R\e_0)+\vol(\Omega_X)((t+R\e_0)^3-t^3), \\
		-E_X(t)\leq -E^{\e_0}_X(t-R\e_0)+\vol(\Omega_X)(t^3-(t-R\e_0)^3),
	\end{gather} for any $\e_0>0$. Choosing $\e_0\defeq 2/t$ we get
	\begin{gather}
		\abs{E_X(t)}\leq\max\left( \abs{E_X^{\e_0}(t+R\e_0)+\BigO(t)} , \abs{E_X^{\e_0}(t-R\e_0)+\BigO(t)} \right) \\
		\ll \abs{E_X^{\e_0}(t+R\e_0)} + \abs{E_X^{\e_0}(t-R\e_0)} + t.
	\end{gather}
	The asymptotic constant depends on the determinant of $X$, but if we restrict $X$ to the compact set $L_0'$ (see \eqref{lop}), then the determinant of $X$ is bounded by a constant which only depends on the fixed set $L_0$.
	By \eqref{decomposition} we have
	\begin{gather}
		\int_{L_0}\abs{E_X(t)}^2\dmu\mu X \ll
		\int_{[1,2)^3}\int_{\R^3} \abs{E_{AN}(t)}^2 \psi(a)\dd a\dd \eta \\
		\ll
		\int_{[1,2)^3}\int_{\R^3} \abs{E_X^{\e_0}(t+R\e_0)}^2 \psi(a)\dd a\dd \eta +
		\int_{[1,2)^3}\int_{\R^3} \abs{E_X^{\e_0}(t-R\e_0)}^2 \psi(a)\dd a\dd \eta +
		t^2,
	\end{gather}
	and noting that $\e_0\geq 1/(t+R\e_0)$ and $\e_0 \geq 1/(t-R\e_0)$ for all sufficiently large $t\pm R\e_0$, the hypothesis \eqref{hypothesis} implies that the right-hand side above is
	\begin{gather}
		\lll
		(t+R\e_0)^2+(t-R\e_0)^2+t^2 \ll t^2,
	\end{gather}
	and thus $\sqrt{\int_{L_0}\abs{E_X(t)}^2\dmu\mu X} \lll t$ follows.
\end{proof}

For the remainder of the section we will assume that $\e\geq 1/t$ for all sufficiently large $t$. We will now estimate the behavior of $E_X^\e$. Consider the Fourier transform of the characteristic function $\chi_{\Omega}$ of the standard unit ball $\Omega$ in $\R^3$. Taking advantage of the fact that $\chi_\Omega$ is a radial function and hence that its Fourier transform is radial as well, an easy calculation shows that (see equation 10 in chapter 6.4 in \cite{SteinShakarchi})
\[\hat{\chi_\Omega}(k)=\dfrac2{\|k\|}\int_0^1 \sin(2\pi \|k\|r)r\dd r,\]
which can be integrated by parts to get
\[\hat{\chi_\Omega}(k)=-\dfrac{\cos(2\pi\|k\|)}{\pi\|k\|^2}+\dfrac{\sin(2\pi\|k\|)}{2\pi^2\|k\|^3}.\]
Since $\Omega_X=X\inv \cdot \Omega$ we get
\begin{gather}
	\hat{\chi_{\Omega_X}}(k)=\int_{X\inv \cdot\Omega}e^{2\pi i x\cdot k}\dd x
	=\int_{ \Omega} e^{2\pi i X\inv y\cdot k}\abs{\det {X\inv}}\dd y \\
	=\abs{\det {X\inv}}\hat{\chi_\Omega}((X\inv)\T k) =
	\abs{\det X}\inv\left(-\dfrac{\cos(2\pi\|k\|_X)}{\pi\|k\|_X^2}+\dfrac{\sin(2\pi\|k\|_X)}{2\pi^2\|k\|_X^3}\right),
\end{gather}
recalling the definition
\[\|k\|_X = \|(X\inv)\T k\|.\]
Recall that $E_X^\e(t)=\sum_{k\neq(0,0,0)}\hat{\chi_{t\Omega_X}}(k)\hat{\rho_\e}(k)$. It is straightforward to show that $\hat{\chi_{t\Omega_X}}(k) = t^3\hat{\chi_{\Omega_X}}(t k)$ and $\hat{\rho_\e}(k) = \hat\rho(\e k)$. Hence we can write
	\begin{gather}
		E_X^\e(t)= S_1+S_2\defeq \\
	-\abs{\det X}\inv t\sum_{k\neq(0,0,0)}\dfrac{\cos(2\pi\|tk\|_X)}{\pi\|k\|_X^2}\hat\rho(\e k)+
	\abs{\det X}\inv \sum_{k\neq(0,0,0)}\dfrac{\sin(2\pi\|tk\|_X)}{2\pi^2\|k\|_X^3}\hat\rho(\e k),
	\end{gather}
	where both sums $S_1,S_2$ are real since $\hat\rho$ is real-valued.
	For $X=AN, A\in\mathcal A,N\in\mathcal N^+,$ we have $\abs{\det X}\inv \ll 1$, so for such $X$ we get
	\begin{gather}
		\abs{S_2}\ll \sum_{k\neq (0,0,0)}\dfrac{\abs{\hat\rho(\e k)} }{\norm k^3}.
	\end{gather}
	We use the fact that $\abs{\hat \rho(\e k)}$ decreases as $1/\|\e k\|^N \leq t^N/\|k\|^N$ for any $N>0$, provided that $\e\geq 1/t$. Then we get $\abs{S_2}\ll \sum_{k\neq0}t^N/\|k\|^{3+N}=t^N\sum_{k\neq0}1/\|k\|^{3+N}\ll t^N$, where the final sum converges to a constant by integral comparison for any $N>0$. Choosing $N=1/2$ gives us $\abs{S_2}\ll t^{1/2}$.

	Consequently we have
	\begin{gather}
		\abs{E_X^\e(t)}^2 = (S_1+S_2)^2 \ll S_1^2+S_2^2 \ll S_1^2+t,
	\end{gather}
	and thus, to prove Theorem \ref{bam} for $n=3$, by Claim \ref{vroom} it will suffice to prove that $\int_{[1,2)^3}\int_{\R^3} S_1^2\psi(a)\dd a\dd\eta\lll t^2$, where
	\begin{gather}
		S_1^2 = \abs{\det X}^{-2} t^2 \sum_{k,l\neq (0,0,0)}\dfrac{\cos(2\pi\|tk\|_X)\cos(2\pi\|tl\|_X)}{\pi^2\|k\|_X^2\|l\|_X^2}\hat\rho(\e k)\hat\rho(\e l)
	\end{gather}
	and $X=AN$, using the parametrization \eqref{param}.
	Write the product $\cos(2\pi\|tk\|_X)\cos(2\pi\|tl\|_X)$ as $(e^{\alpha}+e^{-\alpha})(e^{\beta}+e^{-\beta})/4=\frac14(e^{\alpha+\beta}+e^{\alpha-\beta}+e^{-\alpha+\beta}+e^{-\alpha-\beta})$ where $\alpha\defeq 2\pi i t \|k\|_X$ and $\beta\defeq 2\pi i t\|l\|_X$. We split the integral into a sum of four integrals and treat each case separately, that is, we will prove
	\[t^2\int_{[1,2)^3}\int_{\R^3}\sum_{k,l\neq(0,0,0)}\abs{\det A}^{-2} \dfrac{e^{2\pi it\Phi_{k,l}(AN)}}{4\pi^2\|k\|_{AN}^2\|l\|_{AN}^2}\hat\rho(\e k)\hat\rho(\e l)\psi(a)\dd a\dd \eta \lll t^2\]
	where $\Phi_{k,l}(X)=\pm \|k\|_X\pm\|l\|_X$, for all four different combinations of sign choices.

We cancel the factor $t^2$ on both sides and exchange the order of integration and summation (noting that the sum is uniformly convergent by the rapid decay of $\hat\rho$). 
 Thus, recalling that $\psi(a)=4\pi\abs{\det A}^2\psi_1(a_1)\psi_2(a_2)\psi_3(a_3)$, we arrive at:
\begin{Claim} \label{afterstuff}
	To prove Theorem \ref{bam} for $n=3$ it suffices to prove that
	\begin{gather}
		\sum_{k,l\neq(0,0,0)} \dfrac{\abs{\hat\rho(\e k)\hat\rho(\e l)} }{\|k\|^{2}\|l\|^{2}} \abs{I_{k,l}(t)} \lll 1, \label{cloudberry}
	\end{gather}
	for all $\e=\e(t)$ such that $\e\geq 1/t$ for all sufficiently large $t$, where
	\begin{gather}
		I_{k,l}(t) \defeq \int_{[1,2)^3}\int_{\R^3} e^{2\pi i t\Phi_{k,l}(AN)}\psi_{k,l}(AN) \dd a\dd \eta, \label{Ikl} \\
		\Phi_{k,l}(AN)\defeq \pm \|k\|_{AN}\pm \|l\|_{AN}, \notag \\
		\psi_{k,l}(AN)\defeq \left(\dfrac{\|k\|}{\|k\|_{AN}}\right)^2 \left(\dfrac{\|l\|}{\|l\|_{AN}}\right)^2\psi_1(a_1)\psi_2(a_2)\psi_3(a_3), \notag
	\end{gather}
	for all four choices of signs in the definition of $\Phi_{k,l}$.
\end{Claim}

Consider $\Phi_{k,l}(AN)$ for $A\in \mathcal A, N\in\mathcal N^+$.
Write $\moo k\defeq (N\inv)\T k$ and $\moo l\defeq (N\inv)\T l$. Then
	$\|k\|_{AN}=\|(A\inv)\T(N\inv)\T k\|=\|A\inv \moo k\|$. Similarly $\|l\|_A=\|A\inv \moo l\|$. Using the parametrization \eqref{param}, we get
	\begin{align}
	 	A\inv =\mat{\sqrt{a_1}&0&0\\0&\sqrt{a_2}&0\\0&0&\sqrt{a_3}}, & &
	 	(N\inv)\T=\mat{1&0&0\\-\eta_1&1&0\\\eta_1 \eta_3-\eta_2&-\eta_3&1}
	\end{align}
	and therefore
	\begin{align}
		\Phi_{k,l}(AN)=\pm\sqrt{a_1\moo k_1^2+a_2\moo k_2^2+a_3\moo k_3^2}
		              \pm\sqrt{a_1\moo l_1^2+a_2\moo l_2^2+a_3\moo l_3^2}. \label{fi}
	\end{align}
	where $\moo k_i^2$ denotes the square of the $i$th component of the vector $\moo k=(N\inv)\T k$, and similarly for $\moo l_i^2$. Note that our choice of parametrization \eqref{param} of the entries of $A$ turned the expressions inside the square roots in the exponent $\Phi_{k,l}(AN)$ into linear forms of $a_1,a_2,a_3$.

Since $\norm {X\inv}_{\text{op}}\norm k\leq \norm {Xk} \leq \norm X_{\text{op}}\norm k$ where $\norm X_{\text{op}}$ is the operator norm of the matrix $X$ for any $X$, it follows that $\norm k_{AN}\ll \norm k\ll \norm k_{AN}$ and likewise for $l$, when $AN\in\mathcal A\cdot\mathcal N^+$. Hence $\psi_{k,l}(AN)$ can be bounded above and below by constants uniform in $k$ and $l$ (but depending on $L_0$), and thus $\abs{I_{k,l}(t)}\ll \int\abs{\psi_{k,l}}\ll 1$.

We now show that we may neglect the terms in the sum \eqref{cloudberry} for which either $\norm k$ or $\norm l$ is large, where the notion of ``large'' is given by the following definition.

\begin{Definition}
	We set $\BOUND\defeq 32t(\log t)^2$ for all $t>1$. Note that $\BOUND\lll t$ and $\log(\BOUND)\lll 1$. 
\end{Definition}

\begin{Lemma} \label{crude}
	Assuming that $\e\geq 1/t$ for all sufficiently large $t$, we have
	\[ \sum_{\substack{k,l\neq(0,0,0)\\ \|k\|\geq \BOUND \text{ or }\|l\|\geq \BOUND }} \dfrac{\abs{\hat\rho(\e k)\hat\rho(\e l)} }{\|k\|^{2}\|l\|^{2}} \abs{I_{k,l}(t)} \lll 1\]
	where the analogous bound holds if we interchange $k$ and $l$.
\end{Lemma}
\begin{proof}
	It suffices to bound the sum
	\begin{gather}
		\sum_{\substack{k,l\neq(0,0,0)\\ \norm k\geq \BOUND }} =
		\sum_{\substack{k,l\neq(0,0,0)\\ \norm k,\norm l \geq \BOUND }} +
		\sum_{\substack{k,l\neq(0,0,0)\\ \norm k \geq \BOUND > \norm l }}.	\label{raspberry}
	\end{gather}
	Using the bounds $\abs{I_{k,l}(t)}\ll 1$, $\abs{\hat{\rho}(\e l)}\ll 1$, and finally $\abs{\hat\rho(\e k)} \ll e^{-\sqrt{\norm{\e k}}}$ from \eqref{chocolate}, and assuming that $\e\geq 1/t$, the second sum on the right above can be written as
	\begin{gather}
		\sum_{\substack{k,l\neq(0,0,0)\\ \norm k \geq \BOUND > \norm l }} \dfrac{\abs{\hat\rho(\e k)\hat\rho(\e l)} }{\|k\|^{2}\|l\|^{2}} \abs{I_{k,l}(t)} \ll
		\sum_{\substack{l\neq(0,0,0)\\ \norm l\leq  \BOUND  }} 1
		\sum_{\substack{k\neq(0,0,0)\\ \norm k \geq \BOUND  }}    \dfrac{ e^{-\sqrt{\norm{k/t}}}  }{\|k\|^{2}}. \label{peanuts}
	\end{gather}
	The first sum on the right-hand side of \eqref{peanuts} is
	\begin{gather}
		\ll \int_1^{\BOUND}r^2\dd r \ll \BOUND^3 \lll t^3.
	\end{gather}
	The second sum on the right-hand side of \eqref{peanuts} is
	\begin{gather}
		\ll \int_{\BOUND/2}^\infty e^{-\sqrt{r/t}} \dd r \ll
		\left.\left(-2te^{-\sqrt{r/t}}\left(\sqrt{r/t}+1\right)  \right)\right|_{r=\BOUND/2 } \ll \\
				te^{-\sqrt{16(\log t)^2}}\sqrt{16(\log t)^2} \ll
				te^{-4\log t}(\log t)^2 = t^{-3}(\log t)^2.
	\end{gather}
	Thus the right-hand side of \eqref{peanuts} is
	\begin{gather}
		\lll t^3\cdot t^{-3} (\log t)^2 \lll 1.
	\end{gather}
	The first sum on the right-hand side of \eqref{raspberry} can be written as
	\begin{gather}
		\sum_{\substack{k,l\neq(0,0,0)\\ \norm k,\norm l \geq \BOUND }} \dfrac{\abs{\hat\rho(\e k)\hat\rho(\e l)} }{\|k\|^{2}\|l\|^{2}} \abs{I_{k,l}(t)} \ll
		\sum_{\substack{k\neq(0,0,0)\\ \norm k \geq \BOUND  }}    \dfrac{ e^{-\sqrt{\norm{k/t}}}  }{\|k\|^{2}}
		\sum_{\substack{l\neq(0,0,0)\\ \norm l \geq \BOUND  }}    \dfrac{ e^{-\sqrt{\norm{l/t}}}  }{\|l\|^{2}},
	\end{gather}
	which by our previous calculation is $\ll (t^{-3}(\log t)^2)^2 \ll 1$.
\end{proof}

\begin{Remark}
	If one only wants to prove a weaker version of Theorem \ref{bam} with a bound of the form $\BigO(t^{(n-1)/2+\delta})$ for some $\delta>0$, with no log factors, it suffices to take $\BOUND=t^{1+\delta'}$ for some sufficiently small $\delta'>0$, and to use the elementary estimate $\hat{\rho}(x)\ll 1/\norm x^N$, $N>0$ for the Fourier transform of $\rho$ in the proof of Lemma \ref{crude}.
\end{Remark}

The lemma above shows that we may restrict ourselves to summing only over the integer vectors $k,l\neq (0,0,0)$ bounded in norm by $\BOUND$, and thus it is enough to sum over $k,l\neq (0,0,0)$ such that $\abs{k_i},\abs{l_j}\leq \BOUND$ for all $i,j\in\{1,2,3\}$.
Thus we have:

\begin{Claim}
	To prove Theorem \ref{bam} for $n=3$ it suffices to prove that
	\begin{gather}
		\sum_{\substack{k,l\in \somevectors }} \dfrac{1}{\|k\|^{2}\|l\|^{2}} \abs{I_{k,l}(t)} \lll 1 \label{boundthis}	
	\end{gather}
	where the sum extends over all nonzero integer vectors $k,l\in\Z^3$ with entries bounded by $\BOUND$.
\end{Claim}

\p {Neglecting integer vectors with vanishing coordinates}
In order to bound the sum on the left-hand side of \eqref{boundthis}, we will need to take advantage of nontrivial bounds of the oscillating integral $I_{k,l}(t)$. We will derive such a bound in Section \ref{sectionofmainproof}, but for technical reasons, in order to use that bound, we need the first two coordinates of $k$ and $l$ to be nonzero. In the present section, we will prove that we can neglect the part of the sum where some of $k_1,k_2,l_1,l_2$ are zero.

We begin by showing that the terms for which both some coordinate of $k$ and some coordinate of $l$ is zero can be neglected:

\begin{Lemma} 
	We have
		\[ \sum_{\substack{ k,l\in \somevectors  \\k_1=l_1=0}} \dfrac{1}{\|k\|^{2}\|l\|^{2}} \abs{I_{k,l}(t)} \lll 1.\]
	The same bound holds if we exchange $k_1$ for any other component of $k$, and $l_1$ for any other component of $l$.
\end{Lemma}
\begin{proof}
	We use the trivial bound $\abs{I_{k,l}(t)}\ll 1$ and split the sum into one over $k$ and one over $l$. The sum over $k$ satisfies
	\begin{gather}
		\sum_{\substack{ k\in \somevectors \\k_1=0}} \dfrac{1}{\|k\|^{2}} =
		\sum_{\substack{\abs{k_2},\abs{k_3}\leq \BOUND\\(k_2,k_3)\neq(0,0)}} \dfrac{1}{\|(k_2,k_3)\|^{2}}\ll \int_1^{ \BOUND}\dfrac1{r^2}r\dd r\ll\log(\BOUND)\lll 1
	\end{gather}
	where in the second sum we are only summing over integer vectors in $\Z^2$. The same bound holds for the sum over $l$, so the statement of the lemma follows.
\end{proof}

\pagebreak
We now need a lemma on oscillating integrals; see the corollary of Proposition 2 in chapter VIII in \cite{tjockstein}.

\begin{Lemma}[van der Corput lemma] \label{sslemma}
	Let $\phi,\psi_0:[a,b]\to\R$ be smooth functions defined on some interval $[a,b]$, and suppose that $\phi'$ is monotonic and that there exists a constant $c_0>0$ such that $\phi'(x)\geq c_0$ for all $x$. Then
	\begin{gather}
		\abs{\int_a^b e^{it\phi(x)}\psi_0(x)\dd x} \leq \dfrac C{c_0 t}  \left(\abs{\psi_0(b)} + \int_a^b\abs{\psi_0'(x)}\dd x  \right)
	\end{gather}
	for all $t>0$, where $C$ is an absolute constant. 
\end{Lemma}

We prove in the following two lemmas that we can also neglect the terms for which precisely one of $k$ and $l$ has a zero in the first two coordinates.\footnote{This does not imply an analogous statement for the third coordinate because the proof depends on a bound of the integral $I_{k,l}(t)$, and our choice of decomposition $\mathcal{KAN^+}$ of our integration domain is not symmetric in the coordinates.}

\begin{Lemma} \label{k1lemma} We have
		\[ \sum_{\substack{ k,l\in \somevectors \\ k_1=0 \\ l_1,l_2,l_3\neq 0}} \dfrac{1}{\|k\|^{2}\|l\|^{2}} \abs{I_{k,l}(t)} \lll 1.\]
	The same bound holds if we exchange the roles of $k$ and $l$.
\end{Lemma}

\begin{proof}
	Assume that $k_1=0, k\neq(0,0,0)$ and $l_1,l_2,l_3\neq 0$. Consider $\Phi_{k,l}(AN)$, given by equation \eqref{fi}. The partial derivative with respect to $a_1$ is
	\begin{align}
		\ppx{a_1}\Phi_{k,l}(AN)=\pm\dfrac{\moo k_1^2}{2\|k\|_{AN}}  \pm\dfrac{\moo l_1^2}{2\|l\|_{AN} }.
	\end{align}
	Now, since $\moo k_1=k_1=0$ and $\moo l_1=l_1\neq 0$, we get
	\begin{align}
		\ppx{a_1}\Phi_{k,l}(AN)= \pm\dfrac{\moo l_1^2}{2\|l\|_{AN} }\gg \dfrac{l_1^2}{\|l\|}\gg \dfrac{\abs{l_1}}{\|l\|}.
	\end{align}
	Moreover, the second derivative with respect to $a_1$ is
	\begin{gather}
		\left(\ppx{a_1}\right)^2\Phi_{k,l}(AN)= \mp\dfrac{l_1^4}{4\|l\|_{AN}^3 },
	\end{gather}
	which is either always positive or always negative, depending on the sign $\pm$ in the definition of $\Phi_{k,l}$. Thus the map $\phi(a_1)\defeq \Phi_{k,l}(AN)$ for fixed $a_2,a_3$ is such that $\abs{\phi'(a_1)}\gg \abs{l_1}/\norm l$ and $\phi'$ is monotonic. Writing $[b_1,b_2]$ for the support of the characteristic function $\psi_1$, we can apply the van der Corput Lemma \ref{sslemma} to the integral
	\begin{gather}
		\int_{b_1}^{b_2}e^{2\pi i t \Phi_{k,l}(AN)}\psi_0(a_1) \dd{a_1}
	\end{gather}
	where we have defined $\psi_0(a_1)\defeq \frac{\norm k^2\norm l^2}{\norm k_{AN}^2\norm l_{AN}^2}$. The function $\psi_0$ is bounded since $\norm{k}_{AN}\gg\norm k$ and $\norm l_{AN}\gg\norm l$. Its derivative, by the assumption that $\moo k_1=k_1=0, \moo l_1=l_1\neq 0$, is
	\begin{gather}
		\psi_1'(a_1)=\ddx{a_1} \dfrac{\norm k^2\norm l^2}{(a_1\moo k_1^2+a_2\moo k_2^2+a_3\moo k_3^2)(a_1\moo l_1^2+a_2\moo l_2^2+a_3\moo l_3^2)} \\
		= -\dfrac{\norm k^2\norm l^2 l_1^2}{(a_1\moo k_1^2+a_2\moo k_2^2+a_3\moo k_3^2)(a_1\moo l_1^2+a_2\moo l_2^2+a_3\moo l_3^2)^2} 
		= -\dfrac{\norm k^2\norm l^2}{ \norm k_{AN}^2\norm l_{AN}^2 } \dfrac{l_1^2}{\norm l_{AN}^2},
	\end{gather}
	which is also bounded. Thus the van der Corput Lemma gives us the bound
	\begin{gather}
		\abs{\int_{\R} e^{2\pi i t \Phi_{k,l}(AN)}\psi_{k,l}(AN)\dd{a_1}} \ll \dfrac1t\dfrac{\norm l}{\abs{l_1}},
	\end{gather}
	where the asymptotic constant is independent of $k,l$. Integrating in the rest of the variables yields by compactness
	\begin{gather}
		\abs{I_{k,l}(t)}\ll \int_{[1,2)^3}\int_{\R^2} \dfrac1t\dfrac{\norm l}{\abs{l_1}} \psi_2(a_2)\psi_3(a_3)\dd{a_2}\dd{a_2}\dd\eta \ll \dfrac1t\dfrac{\norm l}{\abs{l_1}}.
	\end{gather}

	Using this bound, it now follows that
		\begin{gather}
			\sum_{\substack{ k,l\in \somevectors \\ k_1=0 \\ l_1,l_2,l_3\neq 0}} \dfrac{1}{\|k\|^{2}\|l\|^{2}} \abs{I_{k,l}(t)} \ll 
		\dfrac1t \sum_{\substack{ k\in \somevectors \\ k_1=0}} \dfrac1{\|k\|^{2}} \sum_{\substack{l\in \somevectors \\ l_1,l_2,l_3\neq 0  }} \dfrac1{\|l\|\abs{l_1}}.
		\end{gather}
	The sum over $k$ has logarithmic behavior in $\BOUND$ since we are summing over a two-dimensional space. We will split the sum over $l$ into one over $l_1$, and one over $(l_2,l_3)$. We have $\|l\|\geq \|(0,l_2,l_3)\|\geq\|(l_2,l_3)\|$, so
		\begin{gather}
			\sum_{ \substack{ k,l\in \somevectors\\ k_1=0\\ l_1,l_2,l_3\neq 0}} \dfrac{1}{\|k\|^{2}\|l\|^{2}} \abs{I_{k,l}(t)} \lll
		\dfrac1t \sum_{1\leq |l_1|\leq \BOUND}\abs{l_1}^{-1} \sum_{1\leq |l_2|,|l_3|\leq \BOUND}\|(l_2,l_3)\|^{-1} \notag \\
		\ll \dfrac1t  \int_1^{\BOUND}\dfrac1x\dd x \int_1^{\BOUND} \dfrac1r r\dd r \ll
		\dfrac1t \cdot \log(\BOUND)\cdot \BOUND\lll 1. \label{k1}
		\end{gather}
	This completes the proof that the sum over $k_1=0$ can be neglected.
\end{proof}

\begin{Lemma} \label{k2lemma}
	We have
		\[ \sum_{\substack{  k,l\in \somevectors \\ k_2=0 \\ l_1,l_2,l_3\neq 0}} \dfrac{1}{\|k\|^{2}\|l\|^{2}} \abs{I_{k,l}(t)} \lll 1.\]
	The same bound holds if we exchange the roles of $k$ and $l$.
\end{Lemma}

\begin{proof}
	Assume that $k_2=0, k\neq(0,0,0)$ and $l_1,l_2,l_3\neq 0$. 	We write
	\begin{gather}
		\sum_{\substack{ k,l\in \somevectors \\ k_2=0 \\ l_1,l_2,l_3\neq 0}} \dfrac{1}{\|k\|^{2}\|l\|^{2}} \abs{I_{k,l}(t)}=\\
		\int_{[1,2)^3} \sum_{\substack{ k,l\in \somevectors \\ k_2=0 \\ l_1,l_2,l_3\neq 0}} \dfrac{1}{\|k\|^{2}\|l\|^{2}} \abs{\int_{\R^3}e^{2\pi i\Phi_{k,l}(AN)}\psi_{k,l}(AN)\dd a}\dd \eta. \label{acrid}
	\end{gather}
	We will split the latter sum into two parts: one in which $\abs{l_2-2\eta_1 l_1}\geq 1$, and one in which $\abs{l_2-2\eta_1 l_1}<1$. We will bound the sum over $\abs{l_2-2\eta_1 l_1}\geq 1$ by mimicking the proof of Lemma \ref{k1lemma}, with the difference that we consider instead  the directional derivative of $\Phi_{k,l}(AN)$ with respect to the direction $(-\eta_1^2,1,0)$.

	We deal first with the part of the sum \eqref{acrid} with $\abs{l_2-2\eta_1 l_1}\geq 1$.
	We change the order of integration inside the integral $I_{k,l}(t)$ such that the innermost integral is taken with respect to $a_2$, and perform a one-variable substitution from $a_2$ to $u\defeq -\eta_1^2a_1+a_2$ inside this integral. 
	Recalling the expression \eqref{fi}, it now follows, since $k_2=0$, that
	\begin{gather}
		\ppx u \Phi_{k,l}(AN)=
				-\eta_1^2 \ppx{a_1}\Phi_{k,l}(AN)+\ppx{a_2}\Phi_{k,l}(AN)
		=\pm\dfrac{-\eta_1^2\moo k_1^2+\moo k_2^2}{2\|k\|_{AN}} \pm\dfrac{-\eta_1^2\moo l_1^2+\moo l_2^2}{2\|l\|_{AN} }\\
		=\pm\dfrac{-\eta_1^2 k_1^2+(-\eta_1 k_1+k_2)^2}{2\|k\|_{AN} } \pm\dfrac{-\eta_1^2 l_1^2+(-\eta_1 l_1+l_2)^2}{2\|l\|_{AN} }\\ 
		=\pm\dfrac{-\eta_1^2 l_1^2+(-\eta_1 l_1+l_2)^2}{2\|l\|_{AN} }=\pm\dfrac{-2\eta_1 l_1 l_2+l_2^2}{2\|l\|_{AN} }
		=\pm\dfrac{l_2(l_2-2\eta_1 l_1)}{2\|l\|_{AN}}
	\end{gather}
	and
	\begin{gather}
		\left(\ppx u\right)^2 \Phi_{k,l}(AN) =  \mp\dfrac{( l_2(l_2-2\eta_1 l_1))^2}{4\|l\|_{AN}^3}.
	\end{gather}
	Whenever $\abs{l_2-2\eta_1 l_1}\geq 1$ holds, we get a bound of the form $\abs{\ppx u\Phi_{k,l}} \gg \abs{l_2}/\|l\|$ with $u\mapsto \ppx u \Phi_{k,l}$ monotonic.
	Since $\psi_1\psi_2$ is the characteristic function of a rectangle, it follows that the support of $u\mapsto \psi_{k,l}(AN)$ is some interval $[b_1,b_2]$, which is bounded in length (independent of $k$ and $l$). The function $u\mapsto \psi_{k,l}(AN)$ restricted to the interval $[b_1,b_2]$ coincides with the function $u\mapsto \frac{\norm k^2\norm l^2}{\norm k_{AN}^2\norm l_{AN}^2}$ because $\psi_1\psi_2\psi_3$ is a characteristic function. The function $u\mapsto \psi_{k,l}(AN)$  is bounded, and so is
	\begin{gather}
		\ppx u \psi_{k,l}(AN)=\ppx u\dfrac{\norm k^2\norm l^2}{\norm k_{AN}^2\norm l_{AN}^2}=\\
		 	-\dfrac{\norm k^2\norm l^2 }{2\norm k_{AN}^2\norm l_{AN}^2} \cdot \dfrac{(-\eta_1^2\moo k_1^2+\moo k_2^2)}{\norm k_{AN}^2} 
		 	-\dfrac{\norm k^2\norm l^2 }{2\norm k_{AN}^2\norm l_{AN}^2} \cdot \dfrac{(-\eta_1^2\moo l_1^2+\moo l_2^2)}{\norm l_{AN}^2} 
	\end{gather}
	on the interval $[b_1,b_2]$
	since $\left|-\eta_1^2\moo l_1^2+\moo l_2^2\right|\ll \|\moo l\|^2\ll \norm l_{AN}^2$ and $-\eta_1^2\moo k_1^2+\moo k_2^2=0$.
	Thus, whenever $\abs{l_2-2\eta_1 l_1}\geq 1$ holds, the van der Corput Lemma \ref{sslemma} gives us the bound
	\begin{gather}
		 \abs{\int_{\R}e^{2\pi i t \Phi_{k,l}(AN)}\psi_{k,l}(AN)\dd {a_2}} =
		\abs{\int_{\R}e^{2\pi i t \Phi_{k,l}(AN)}\psi_{k,l}(AN)\dd u} \ll \dfrac1t\dfrac{\norm l}{\abs{l_2}},
	\end{gather}
	and estimating trivially in the remaining variables $a_1,a_3$ yields
	\begin{gather}
		\abs{\int_{\R^3}e^{2\pi i t\Phi_{k,l}(AN)}\psi_{k,l}(AN)\dd a} \ll \dfrac1t\dfrac{\norm l}{\abs{l_2}}. \label{lasse}
	\end{gather}

	This bound yields
	\begin{gather}
		\int_{[1,2)^3} \sum_{\substack{ k,l\in \somevectors \\ k_2=0 \\ l_1,l_2,l_3\neq 0 \\ \abs{l_2-2\eta_1 l_1}\geq 1 }} \dfrac{1}{\|k\|^{2}\|l\|^{2}} \abs{\int_{\R^3}e^{2\pi i\Phi_{k,l}(AN)}\psi_{k,l}(AN)\dd a}\dd \eta \ll \\
		\dfrac1t \int_{[1,2)^3} \sum_{\substack{ k,l\in \somevectors \\ k_2=0 \\ l_1,l_2,l_3\neq 0 \\ \abs{l_2-2\eta_1 l_1}\geq 1 }} \dfrac{1}{\|k\|^{2}\|l\|^{1}\abs{l_2}}\dd \eta \leq 
		\dfrac1t \sum_{\substack{ k,l\in \somevectors \\ k_2=0 \\ l_1,l_2,l_3\neq 0}} \dfrac{1}{\|k\|^{2}\|l\| \abs{l_2}}\lll 1,
	\end{gather}
	where the last bound is completely analogous to the bound \eqref{k1}.

	It remains to bound the part of the sum \eqref{acrid} with $\abs{l_2-2\eta_1 l_1}<1$.
	When $\abs{l_2-2\eta_1 l_1}<1$, there are at most two values that $l_2$ may assume when $\eta_1,l_1$ are held fixed, and using $\|(l_1,l_2,l_3)\|\geq\|(l_1,0,l_3)\|=\|(l_1,l_3)\|$, we get
	\begin{gather}
		\int_{[1,2)^3} \sum_{\substack{ k,l\in \somevectors \\ k_2=0 \\ l_1,l_2,l_3\neq 0 \\ \abs{l_2-2\eta_1 l_1}< 1 }} \dfrac{1}{\|k\|^{2}\|l\|^{2}} \abs{\int_{\R^3}e^{2\pi i\Phi_{k,l}(AN)}\psi_{k,l}(AN)\dd a}\dd \eta \ll \\
		\int_{[1,2)^3} \sum_{\substack{ k,l\in \somevectors \\ k_2=0 \\ l_1,l_2,l_3\neq 0 \\ \abs{l_2-2\eta_1 l_1}< 1 }} \dfrac{1}{\|k\|^{2}\|l\|^{2}} \dd \eta \ll \\
		\int_{[1,2)^3} \sum_{\substack{ k,l\in \somevectors \\ k_2=0 \\ l_1,l_2,l_3\neq 0 \\ \abs{l_2-2\eta_1 l_1}< 1 }} \dfrac{1}{\|(k_1,k_3)\|^{2}\|(l_1,l_3)\|^{2}} \dd \eta \ll \\
		\sum_{1\leq \abs{k_1},\abs{k_3}\leq \BOUND}
		\sum_{1\leq \abs{l_1},\abs{l_3}\leq \BOUND}
		\dfrac{1}{\|(k_1,k_3)\|^{2}\|(l_1,l_3)\|^{2}} \lll 1,
	\end{gather}
	and we are done.
\end{proof}

Putting the lemmas together, we have thus demonstrated:
\begin{Claim} \label{beforefinalsum}
	To prove Theorem \ref{bam} for $n=3$ it suffices to prove that
	\begin{gather}
		\sum_{\substack{ k,l\in \somevectors \\ k_1,k_2,l_1,l_2\neq 0}} \dfrac{1}{\|k\|^{2}\|l\|^{2}} \abs{I_{k,l}(t)} \lll 1,
	\end{gather}
	where $k_3,l_3$ may assume both zero and nonzero values.
\end{Claim}

Proving the inequality in Claim \ref{beforefinalsum} is the heart of the proof of Theorem \ref{bam}; we will do this in the next section.

\p {Concluding the proof of Theorem \ref{bam}} \label{sectionofmainproof}

Recall that $\moo k = (N\inv)\T k, \moo l=(N\inv)\T l$. We now define $\gamma\defeq -\eta_1$. Then we have $\moo k_1=k_1, \moo k_2=\gamma k_1+k_2$ and $\moo l_1=l_1, \moo l_2=\gamma l_1+l_2$, and thus
\begin{gather}
	\begin{split}
		\moo k_1\moo l_2-\moo k_2\moo l_1&=  k_1 l_2- k_2 l_1,\\
		\moo k_1\moo l_2+\moo k_2\moo l_1&= k_1 l_2+ k_2 l_1+2\gamma k_1l_1.
	\end{split} \label{klkl}
\end{gather}
The crucial ingredient in the proof of the inequality in Claim \ref{beforefinalsum} is the following inequality, and the uniformity of the bound is essential, as we will apply it to all terms of an infinite sum.

\begin{Lemma} \label{hessianlemma}
	Assume that $\abs{\moo k_1^2\moo l_2^2-\moo k_2^2\moo l_1^2}\neq 0$. Then
	\begin{gather}
		\abs{\int_{\R^3} e^{2\pi i t \Phi_{k,l}(AN)} \psi_{k,l}(AN)\dd a } \leq \dfrac{C}t \dfrac{\norm{k}^{3/2}\norm{l}^{3/2}}{\abs{\moo k_1^2\moo l_2^2-\moo k_2^2\moo l_1^2}}
	\end{gather}
	for all $t>0$, where $C$ is a constant which does not depend on $k,l,N$ (but which does depend on the already fixed cutoff function $\psi$).
\end{Lemma}

We will postpone the proof of Lemma \ref{hessianlemma} until we need to use it; 
Lemma \ref{hessianlemma} compels us to split the sum in Claim \ref{beforefinalsum} into parts as follows. We write
	\begin{gather}
		\sum_{\substack{ k,l\in \somevectors \\k_1,k_2,l_1,l_2\neq 0  }} \dfrac{1}{\|k\|^{2}\|l\|^{2}} \abs{I_{k,l}(t)} \leq \\
		\int_{[1,2)^3} \left(\sum\nolimits_1+\sum\nolimits_2+\sum\nolimits_3\right) \dfrac{1}{\|k\|^{2}\|l\|^{2}} \abs{\int_{\R^3}e^{2\pi i t\Phi_{k,l}(AN)}\psi_{k,l}(AN)\dd a}\dd \eta, \label{thirdsum}
	\end{gather}
where $\sum_1$ is the sum over $\abs{k_1l_2+k_2l_1+2\gamma k_1 l_1}<1/2$; $\sum_2$ is the sum over $k_1l_2-k_2l_1=0$; $\sum_3$ is the sum over $\abs{k_1l_2+k_2l_1+2\gamma k_1 l_1}\geq 1/2$ and $k_1l_2-k_2l_1\neq 0$, and where all sums range over $k,l\in\somevectors$ such that $k_1,k_2,l_1,l_2\neq 0$.

The following lemma shows that we may neglect the sums $\sum_1$ and $\sum_2$:

\begin{Lemma} \label{sum1and2}
	For any $\abs\gamma\geq 1$, we have
	\[ \sum_{\substack{ k,l\in \somevectors \\ k_1,k_2,l_1,l_2\neq 0\\ \abs{k_1l_2+k_2l_1+2\gamma k_1 l_1}< 1/2}} \dfrac1{\|k\|^2\|l\|^2} \lll 1, \]
	where the asymptotic constant is independent of $\gamma$, and
	\[ \sum_{\substack{ k,l\in \somevectors \\ k_1,k_2,l_1,l_2\neq 0\\ \abs{k_1l_2-k_2l_1}=0}} \dfrac1{\|k\|^2\|l\|^2} \lll 1. \]
\end{Lemma}

\begin{proof}
We obtain the second sum by substituting $k_2\mapsto -k_2$ and $\gamma=0$ in the first sum. Thus it suffices to bound the first sum in the cases $\abs{\gamma}\geq 1$ and $\gamma=0$. We will treat both cases simultaneously. We have
\begin{gather}
	\sum_{\substack{ k,l\in \somevectors \\k_1,k_2,l_1,l_2\neq 0\\ \abs{k_1l_2+k_2l_1+2\gamma k_1 l_1}< 1/2}} \dfrac1{\|k\|^2\|l\|^2} \ll \\
	\sum_{\substack{ k,l\in \somevectors \\ k_1,k_2,l_1,l_2\neq 0\\ \abs{k_1l_2+k_2l_1+2\gamma k_1 l_1}< 1/2}} \dfrac1{(\|(k_1,k_2)\|+\abs{k_3})^2 (\|(l_1,l_2)\|+\abs{l_3})^2} \ll \\
	\sum_{\substack{ 1\leq |k_1|,\abs{k_2},|l_1|,\abs{l_2}\leq \BOUND\\ \abs{k_1l_2+k_2l_1+2\gamma k_1 l_1}< 1/2}} \int_0^{\BOUND}\int_0^{\BOUND} \dfrac{\dd{k_3}\dd{l_3} }{(\|(k_1,k_2)\|+\abs{k_3})^2 (\|(l_1,l_2)\|+\abs{l_3})^2} \ll \\
	\sum_{\substack{ 1\leq |k_1|,\abs{k_2},|l_1|,\abs{l_2}\leq \BOUND\\ \abs{k_1l_2+k_2l_1+2\gamma k_1 l_1}< 1/2}} \dfrac{1}{\|(k_1,k_2)\| \|(l_1,l_2)\|} \leq
	\sum_{\substack{1\leq \abs a, \abs b, \abs x,\abs y\leq \BOUND \\ bx-ay=[2\gamma ab]}} \dfrac1{ \norm{(a,b)} \norm{(x,y)} }\leq \\
	\sum_{r=1}^{\BOUND }\sum_{\substack{ 1\leq\abs a,\abs b\leq \BOUND \\ \gcd(a,b)=1 }}\sum_{\substack{1\leq \abs x,\abs y\leq \BOUND \\ bx-ay=[2\gamma r^2ab]/r}} \dfrac1{ r\norm{(a,b)} \norm{(x,y)} },
\end{gather}
where we have used the notation $[x]$ for the integer nearest to $x\in\R$, where we round away from zero if there is an ambiguity.

\begin{figure}[h]
	\label{helpfulpicture}
	\caption{In the proof of Lemma \ref{sum1and2}, each integer point $(x,y)$ on the line $bx-ay=c$ is mapped to the closest integer point $(x',y')$ on the line $L$ with shorter or equal length.}
	\begin{center}
		\fbox{\includegraphics{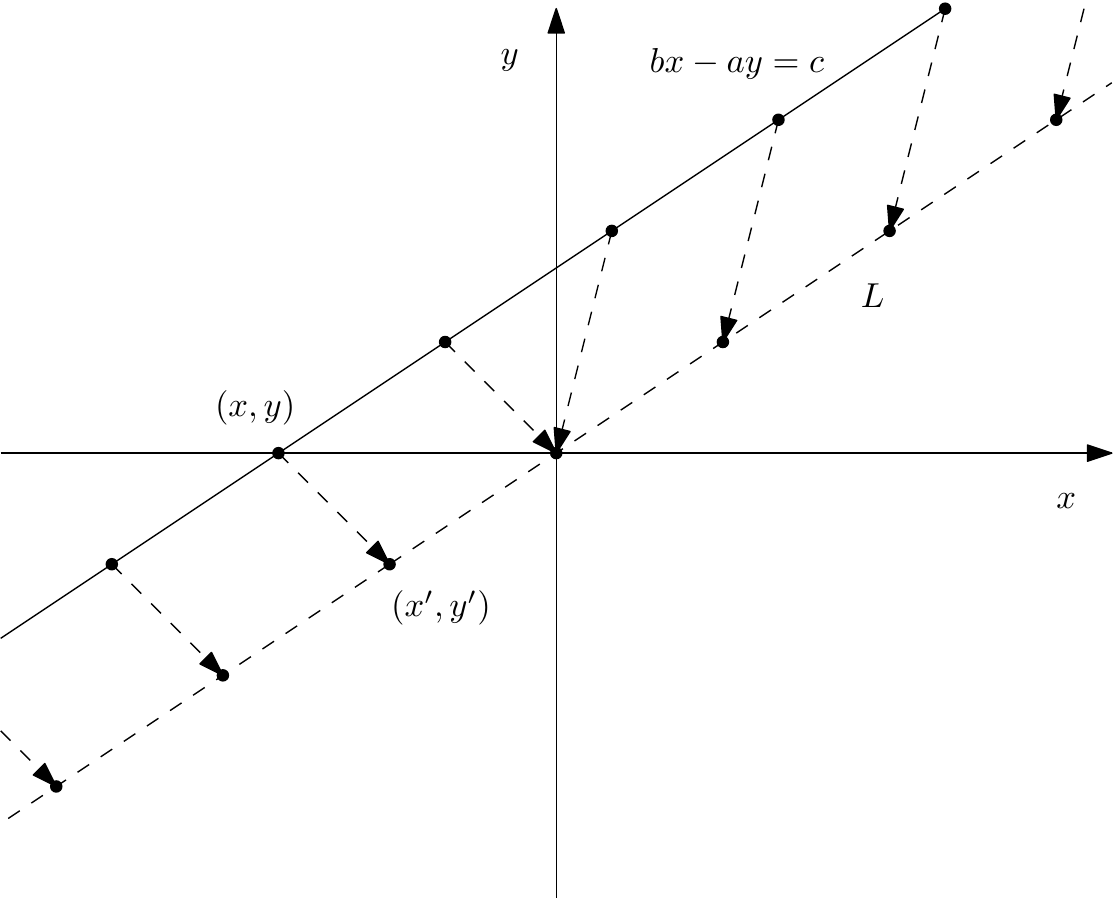}}
	\end{center}
\end{figure}

Consider the innermost sum, in which $a,b,r$ are fixed, and let $c\defeq [2\gamma r^2ab]/r$.
Now, since $\gcd(a,b)=1$, the equation $bx-ay=c$ has the set of solutions $(x,y)=(x_0,y_0)+m(a,b), m\in\Z$, granted there exists some solution $(x_0,y_0)\in\Z^2$. For each solution $(x,y)$ we will define $(x',y')$ to be the integer vector on the line $L$ spanned by $(a,b)$ which is closest to $(x,y)$ among all vectors $(x',y')$ with $\norm{(x',y')}\leq\norm{(x,y)}$; if there is an ambiguity, choose the shorter vector $(x',y')$. See Figure \ref{helpfulpicture}. We see that the set of solutions $(x,y)\in\Z^2$ maps to the set of vectors $(x',y')=m(a,b), m\in\Z$, with at most two vectors $(x,y)$ mapping to any given $(x',y')$. Now we will bound $1/\|(x,y)\|$ by $1/\|(x',y')\|=1/(m\|(a,b)\|)$ if $m\neq 0$, and otherwise we will use the bound $1/\|(x,y)\|\leq 1/D$, where $D$ is the distance between the line $bx-ay=c$ and the origin in $\R^2$. Note that the case $m=0$ cannot occur if $\gamma=0$ since we are summing over nonzero vectors only; but if $\abs \gamma\geq 1$, we get $D=\abs c/\|(a,b)\|\geq \abs{2rab}/\|(a,b)\|$. We also have $\abs m\leq\sqrt 2 \cdot \BOUND$. Thus the last sum above can be bounded by
\begin{gather}
	\sum_{r=1}^{\BOUND }\sum_{\substack{ 1\leq\abs a,\abs b\leq \BOUND \\ \gcd(a,b)=1 }} \left( 2\dfrac1{ r\norm{(a,b)}}\dfrac{\|(a,b)\|}{2rab} + 4\sum_{m=1}^{\sqrt 2 \BOUND} \dfrac1{ r\norm{(a,b)} \norm{m(a,b)} } \right) \ll \\
	\sum_{r=1}^{\BOUND }\sum_{a=1}^{\BOUND}\sum_{b=1}^{\BOUND}\dfrac1{r^2}\dfrac1a\dfrac1b+
	\sum_{r=1}^{\BOUND }\sum_{\substack{ 1\leq\abs a,\abs b\leq \sqrt 2 \BOUND }} \sum_{m=1}^{\sqrt 2 \BOUND} \dfrac 1r\dfrac 1m \dfrac1{ \norm{(a,b)}^2 } \lll 1,
\end{gather}
where all the individual sums in the last expression have at worst logarithmic behavior in $\BOUND$, so we are done.
\end{proof}

It remains to deal with the third part of \eqref{thirdsum}, and for this we will need to use the integral bound from Lemma \ref{hessianlemma}. First let us prove Lemma \ref{hessianlemma}.

\begin{proof}[Proof of Lemma \ref{hessianlemma}]
	We will prove the bound for the inner integral with respect to $a_1$ and $a_2$. Then the result follows by the compactness of the integration domain. Recalling \eqref{fi}, the integral we need to bound is
	\begin{gather}
		\int_{\R^2} \exp\left(2\pi i t \left ( \pm\sqrt{a_1\moo k_1^2+a_2\moo k_2^2+a_3\moo k_3^2}
		                         \pm\sqrt{a_1\moo l_1^2+a_2\moo l_2^2+a_3\moo l_3^2}  \right)\right)\times \\
		     \times \left(\dfrac{\norm k}{\sqrt{a_1\moo k_1^2+a_2\moo k_2^2+a_3\moo k_3^2}}\right)^2
		     \left(\dfrac{\norm l}{\sqrt{a_1\moo l_1^2+a_2\moo l_2^2+a_3\moo l_3^2}}\right)^2 \psi_1(a_1)\psi_2(a_2) \dd{a_1}\dd{a_2}.
	\end{gather}
	We perform a variable substitution from $(a_1,a_2)$ to $(x,y)$ where $x\defeq a_1\moo k_1^2+a_2\moo k_2^2+a_3\moo k_3^2, y\defeq a_1\moo l_1^2+a_2\moo l_2^2+a_3\moo l_3^2$, which yields the Jacobian $1/|\moo k_1^2\moo l_2^2-\moo k_2^2\moo l_1^2|$. The integral above becomes
	\begin{gather}
		\dfrac{1}{|\moo k_1^2\moo l_2^2-\moo k_2^2\moo l_1^2|}
		\int_{\R^2} e^{2\pi i t( \pm\sqrt{x}\pm\sqrt{y}  )} \frac{\norm k^2}{x} \frac{\norm l^2}{y}  \Psi_{k,l,N}(x,y) \dd x\dd y. \label{ulf}
	\end{gather}
	where we define $\Psi_{k,l,N}(x,y)\defeq \psi_1(a_1)\psi_2(a_2)$ (noting that $a_1,a_2$ may be expressed in terms of $x,y$ when $a_3,\eta,k,l$ are held fixed). 
	Since $a_1,a_2,a_3$ are bounded above and below throughout the support of $\psi_1\psi_2\psi_3$, it follows that $\abs x\ll \|\moo k\|^2\ll \norm{k}^2$, and similarly $\abs x \gg \|\moo k\|^2\gg \norm k^2$, throughout the support of $\Psi_{k,l,N}$. Likewise $\abs y\ll\|l\|^2$ and $\abs y\gg \|l\|^2$ throughout the support of $\Psi_{k,l,N}$.

	We will assume without loss of generality that $\norm k\geq \norm l$, and use integration by parts on the inner integral of \eqref{ulf} with respect to $x$; if instead $\norm k\leq\norm l$ were the case, we repeat the following argument but integrate by parts instead with respect to $y$.
	An antiderivative of $e^{2\pi i t\sqrt x}$ with respect to $x$ is $\dfrac{e^{2\pi it\sqrt x}}{\pi it}\left(\sqrt x-\dfrac1{2\pi it}\right)$. Since $\psi_1\psi_2$ is the characteristic function of a rectangle, it follows that $x\mapsto \Psi_{k,l,N}(x,y)$ is the characteristic function of some interval $[b_1(y),b_2(y)]$, where the length of the interval is $\ll \|k\|^2$. Thus
	\begin{gather}
		\int_{\R} e^{2\pi i t( \pm\sqrt{x}\pm\sqrt{y}  )} \frac{\norm k^2}{x} \frac{\norm l^2}{y}  \Psi_{k,l,N}(x,y) \dd x = \\
			\left[
				\dfrac{e^{2\pi i t( \pm\sqrt{x}\pm\sqrt{y}  )}}{\pm \pi it}\left (\sqrt x-\dfrac1{2\pi i t}\right) \frac{\norm k^2}{x} \frac{\norm l^2}{y} 
			\right]_{x=b_1(y)}^{b_2(y)} - \\
				\int_{b_1(y)}^{b_2(y)}
				\dfrac{e^{2\pi i t( \pm\sqrt{x}\pm\sqrt{y}  )}}{\pm \pi it}\left (\sqrt x-\dfrac1{2\pi i t}\right) \left(-\frac{\norm k^2}{x^2}\right) \frac{\norm l^2}{y}  \dd x.
	\end{gather}
	Using the bounds $\norm k^2\ll \abs x\ll\norm k^2$, we can bound the above expression by
	\begin{gather}
		2\sup_{x\in[b_1(y),b_2(y)]}
			\left(
				\dfrac{e^{2\pi i t( \pm\sqrt{x}\pm\sqrt{y}  )}}{\pm \pi it}\left (\sqrt x-\dfrac1{2\pi i t}\right) \frac{\norm k^2}{x} \frac{\norm l^2}{y} 
			\right) + \\
			\abs{b_2(y)-b_1(y)}\times\sup_{x\in[b_1(y),b_2(y)]}\left(
				\dfrac{e^{2\pi i t( \pm\sqrt{x}\pm\sqrt{y}  )}}{\pm \pi it}\left (\sqrt x-\dfrac1{2\pi i t}\right) \left(-\frac{\norm k^2}{x^2}\right) \frac{\norm l^2}{y}
			\right) \ll \\
		\sup_{x\in[b_1(y),b_2(y)]}
			\left(
				\dfrac{1}{t} \sqrt x \frac{\norm k^2}{x} \frac{\norm l^2}{y} 
			\right) + 
			\abs{b_2(y)-b_1(y)}\times\sup_{x\in[b_1(y),b_2(y)]}\left(
				\dfrac{1}{t}\dfrac{\sqrt x}x \frac{\norm k^2}{x} \frac{\norm l^2}{y}
			\right) \ll \\
		\dfrac1t\sqrt{\norm k^2} \frac{\norm k^2}{\norm k^2} \frac{\norm l^2}{y} +
		\norm k^2 \dfrac1t\dfrac1{\sqrt{\norm k^2}} \frac{\norm k^2}{\norm k^2} \frac{\norm l^2}{y} = 
		2\dfrac1t\norm{k} \frac{\norm l^2}{y}.
	\end{gather}
	We finally integrate with respect to $y$, and use the bounds $\norm l^2\ll\abs y\ll \norm l^2$.
	Write $D\defeq \{y\in\R: \Psi_{k,l,N}(x,y)=1\text{ for some $x\in\R$}\}$ for the domain of integration.
	Thus \eqref{ulf} is bounded by
	\begin{gather}
		\dfrac{1}{|\moo k_1^2\moo l_2^2-\moo k_2^2\moo l_1^2|}\norm l^2\sup_{y\in D}\left(\dfrac1t\norm{k} \frac{\norm l^2}{y}\right ) \ll 
		\dfrac{1}{|\moo k_1^2\moo l_2^2-\moo k_2^2\moo l_1^2|} \dfrac{\norm k\norm l^2}{t} = \\
		\dfrac{1}{|\moo k_1^2\moo l_2^2-\moo k_2^2\moo l_1^2|} \dfrac{\norm k^{3/2}\norm l^{3/2}}{t} \dfrac{\norm l^{1/2}}{\norm k^{1/2}} \leq 
		\dfrac{1}{|\moo k_1^2\moo l_2^2-\moo k_2^2\moo l_1^2|} \dfrac{\norm k^{3/2}\norm l^{3/2}}{t},
	\end{gather}
	where the last inequality follows from our assumption $\norm k\geq\norm l$.
\end{proof}

Applying Lemma \ref{hessianlemma}, and recalling \eqref{klkl}, it now only remains to bound
\begin{gather}
	\int_{[1,2)^3} \sum\nolimits_3 \dfrac{1}{\|k\|^{2}\|l\|^{2}} \abs{\int_{\R^3}e^{2\pi i t\Phi_{k,l}(AN)}\psi_{k,l}(AN)\dd a}\dd \eta \ll \\
	\int_{[1,2)^3} \sum\nolimits_3 \dfrac{1}{\|k\|^{2}\|l\|^{2}} \dfrac1{t}  \dfrac{\|k\|^{3/2}\|l\|^{3/2}}{\abs{\moo k_1^2\moo l_2^2-\moo k_2^2 \moo l_1^2} } \dd \eta = \\
	\int_{[1,2)^3} \sum\nolimits_3 \dfrac{1}{\|k\|^{2}\|l\|^{2}} \dfrac1{t}  \dfrac{\|k\|^{3/2}\|l\|^{3/2}}{\abs{k_1 l_2-k_2 l_1} \abs{k_1l_2+k_2l_1+2\gamma k_1l_1}} \dd \eta.
\end{gather}
The integrand only depends on $\eta_1=-\gamma$. Integrating with respect to $\eta_2$ and $\eta_3$, the expression above becomes
\begin{gather}
	\int_{(-2,-1]}
	\sum_{\substack{ k,l\in \somevectors \\k_1,k_2,l_1,l_2\neq 0 \\ \abs{k_1l_2+k_2l_1+2\gamma k_1 l_1}\geq 1/2 \\ k_1l_2-k_2l_1\neq 0 }}
	\dfrac{1}{\|k\|^{2}\|l\|^{2}} \dfrac1{t}  \dfrac{\|k\|^{3/2}\|l\|^{3/2}}{\abs{k_1 l_2-k_2 l_1} \abs{k_1l_2+k_2l_1+2\gamma k_1l_1}} \dd \gamma.
\end{gather}

We split the sum into one over $k_3,l_3$ and one over the other coordinates. We use the fact that $\|k\|\geq \abs{k_3}$ if $k_3\neq 0$, and otherwise $\|k\|\geq 1$, and likewise for $l$. Thus the above expression is bounded by
	\begin{gather}
		\dfrac1{t}
		\int_{-2}^{-1}
		\sum_{\substack{ 1\leq\abs {k_1},\abs{k_2},\abs{l_1},\abs{l_2}\leq \BOUND \\ \abs{k_1l_2+k_2l_1+2\gamma k_1 l_1}\geq 1/2 \\ k_1l_2-k_2l_1\neq 0 }}
		 \dfrac1{\abs{k_1 l_2-k_2 l_1} \abs{k_1l_2+k_2l_1+2\gamma k_1l_1}} \dd\gamma \times \\
		\times \left(1+\sum_{1\leq\abs{k_3}\leq \BOUND}\dfrac{1}{\abs{k_3}^{1/2}}\right)
		\left(1+\sum_{1\leq\abs{l_3}\leq \BOUND} \dfrac{1}{\abs{l_3}^{1/2}} \right)
		 \ll \\
		\dfrac{\left((\BOUND)^{1/2}\right)^2 }{t}
		\int_{-2}^{-1}
		\sum_{\substack{ 1\leq\abs {k_1},\abs{k_2},\abs{l_1},\abs{l_2}\leq \BOUND \\ \abs{k_1l_2+k_2l_1+2\gamma k_1 l_1}\geq 1/2 \\ k_1l_2-k_2l_1\neq 0 }}
		 \dfrac1{\abs{k_1 l_2-k_2 l_1} \abs{k_1l_2+k_2l_1+2\gamma k_1l_1}} \dd\gamma \lll \\
		\int_{-2}^{-1}
		\sum_{\substack{ 1\leq\abs {k_1},\abs{k_2},\abs{l_1},\abs{l_2}\leq \BOUND \\ \abs{k_1l_2+k_2l_1+2\gamma k_1 l_1}\geq 1/2 \\ k_1l_2-k_2l_1\neq 0 }}
		 \dfrac1{\abs{k_1 l_2-k_2 l_1} \abs{k_1l_2+k_2l_1+2\gamma k_1l_1}} \dd\gamma \leq    \label{vom} \\
		 		\int_{-2}^{-1}
		 \sum_{r=1}^{\BOUND}
		\sum_{\substack{ 1\leq\abs {k_1},\abs{k_2},\abs{l_1},\abs{l_2}\leq \BOUND \\ \abs{k_1l_2+k_2l_1+2\gamma k_1 l_1r}\geq 1/(2r) \\ k_1l_2-k_2l_1\neq 0 \\ \gcd(k_1,l_1)=1 }}
		 \dfrac1{r^2 \abs{k_1 l_2-k_2 l_1} \abs{k_1l_2+k_2l_1+2\gamma k_1l_1r}} \dd\gamma \leq \\
		 		\int_{-2}^{-1}
		 \sum_{r=1}^{\BOUND}
		 \sum_{1\leq \abs w\leq 2\BOUND^2 }
		\sum_{\substack{ 1\leq\abs {k_1},\abs{l_1}\leq \BOUND \\ \gcd(k_1,l_1)=1 }}
		\sum_{\substack{ 1\leq\abs{k_2},\abs{l_2}\leq \BOUND \\ \abs{k_1l_2+k_2l_1+2\gamma k_1 l_1r}\geq 1/(2r) \\ k_1l_2-k_2l_1=w }}
		 \dfrac1{r^2 \abs{w} \abs{k_1l_2+k_2l_1+2\gamma k_1l_1r}} \dd\gamma. \label{freml}
	\end{gather}
Consider the innermost sum, where $k_1,l_1,\gamma,w,r$ are fixed. Since $\gcd(k_1,l_1)=1$ inside the sum, it follows that the equation $k_1 l_2-k_2 l_1=w$ has the set of solutions $(k_2,l_2)=(x_0,y_0)+m(k_1,l_1), m\in\Z$, granted there exists some solution $(x_0,y_0)\in\Z^2$. Therefore $k_1l_2+k_2l_1+2\gamma k_1l_1r$ assumes the values $c_0+2k_1 l_1 m$ for $m\in\Z$  as $(k_2,l_2)$ varies, where $c_0\defeq k_1y_0+l_1x_0+2\gamma k_1l_1r$ is constant. In particular, $k_1l_2+k_2l_1+2\gamma k_1l_1r$ assumes consecutive values spaced a distance $2\abs{k_1 l_1}$ apart, with at most two values smaller than $2\abs{k_1 l_1}$ in absolute value, and the number of values it assumes is $\leq 2\BOUND$.
It follows that the expression \eqref{freml} above is
	\begin{gather}
		\ll
		\int_{-2}^{-1}
		 \sum_{r=1}^{\BOUND}
		 \sum_{1\leq \abs w\leq 2\BOUND^2}
		\sum_{\substack{ 1\leq\abs {k_1},\abs{l_1}\leq \BOUND \\ \gcd(k_1,l_1)=1 }}
			\dfrac1{r^2\abs w} \times  \\
			\times \left(\sum_{1\leq \abs m\leq \BOUND}\dfrac{1}{2\abs{mk_1l_1}}+
				\!\!\!\!\!
				\sum_{\substack{ 1\leq\abs{k_2},\abs{l_2}\leq \BOUND \\ \frac1{2r} \leq \abs{k_1l_2+k_2l_1+2\gamma k_1 l_1r} < 2\abs{k_1 l_1} \\ k_1l_2-k_2l_1=w }}
					\!\!\!\!\!\!\!\!\!\!
					\dfrac1{\abs{k_1l_2+k_2l_1+2\gamma k_1l_1r}}
			\right)
			\dd\gamma.
	\end{gather}
We expand this into a sum of two terms. We have
	\begin{gather}
		\int_{-2}^{-1}
		 \sum_{r=1}^{\BOUND}
		 \sum_{1\leq \abs w\leq 2\BOUND^2}
		\sum_{\substack{ 1\leq\abs {k_1},\abs{l_1}\leq \BOUND \\ \gcd(k_1,l_1)=1 }}
			\dfrac1{r^2\abs w} \sum_{1\leq \abs m\leq \BOUND}\dfrac{1}{2\abs{mk_1l_1}}\dd\gamma \lll 1,
	\end{gather}
which takes care of the first term. It remains to bound
	\begin{gather}
		\int_{-2}^{-1}
		\sum_{\substack{ 1\leq r,\abs{k_1},\abs{l_1}\leq \BOUND \\ 1\leq \abs w\leq 2\BOUND^2 \\ \gcd(k_1,l_1)=1 }}
			\dfrac1{r^2\abs w} 
				\!\!\!\!\!
				\sum_{\substack{ 1\leq\abs{k_2},\abs{l_2}\leq \BOUND \\ \frac1{2r} \leq \abs{k_1l_2+k_2l_1+2\gamma k_1 l_1r} < 2\abs{k_1 l_1} \\ k_1l_2-k_2l_1=w }}
					\!\!\!\!\!\!\!\!\!\!
					\dfrac1{\abs{k_1l_2+k_2l_1+2\gamma k_1l_1r}}
			\dd\gamma.
	\end{gather}
We may without loss of generality assume that $k_1l_2+k_2l_1+2\gamma k_1 l_1 r$ is positive in the innermost sum, since we obtain the opposite case by switching the signs of $k_1,k_2,w$. Moreover, we may extend the sum to range over all $(k_2,l_2)\in\Z^2$.
It thus suffices to bound
\begin{gather}
		\sum_{\substack{ 1\leq r,\abs{k_1},\abs{l_1}\leq \BOUND \\ 1\leq \abs w\leq 2\BOUND^2 \\ \gcd(k_1,l_1)=1 }}
			\dfrac1{r^2\abs w} 
				\int_{-2}^{-1}
				\sum_{\substack{ (k_2,l_2)\in\Z^2 \\ \frac1{2r} \leq ({k_1l_2+k_2l_1+2\gamma k_1 l_1r}) < 2\abs{k_1 l_1} \\ k_1l_2-k_2l_1=w }}
					\!\!\!\!\!\!\!\!\!\!
					\dfrac1{({k_1l_2+k_2l_1+2\gamma k_1l_1r})}
			\dd\gamma.
\end{gather}
In the innermost sum, which is a sum over precisely one pair $(k_2,l_2)$, and where $k_1,l_1,\gamma,w,r$ are fixed,  denote by $f(\gamma)$ the unique positive value in $[1/(2r),\abs{2 k_1 l_1})$ which $k_1l_2+k_2l_1+2\gamma k_1l_1r$ assumes as $(k_2,l_2)$ varies, if it exists, or let $f(\gamma)$ be undefined otherwise. Then $f(\gamma)=c+2\gamma k_1l_1r \pmod{2\abs{k_1l_1}}$ on its domain of definition, where $c=k_1y_0+l_1x_0$ is a constant, so $f(\gamma)$ coincides with a sawtooth wave with slope $2k_1l_1r$ and period $1/r$, except that it is undefined where the sawtooth wave has a value in $[0,1/(2r))$. Now we can partition $(-2,1]\cap\op{dom}(f)$ into at most $r+1$ subintervals $I_m$ such that $f$ is linear on each. The integral of $1/f(\gamma)$ with respect to $\gamma$ on any such subinterval $I_m$ is
\begin{gather}
	\int_{I_m} \dfrac{\dd\gamma}{f(\gamma)}
	= \left[ \dfrac{\log\abs{k_1l_2+k_2l_1+2\gamma k_1l_1r} }{2 k_1l_1r} \right]_{\gamma=\op{inf}I_m }^{ \op{sup}I_m } \ll \\
	\dfrac{ \log\abs{(2+4r)\BOUND^2}+\abs{\log\frac1{2r} } }{ \abs{2 k_1l_1r} } \lll \dfrac{ \log r }{ \abs{k_1l_1r} },
\end{gather}
where the asymptotic constants are independent of $m$.
We now get
	\begin{gather}
		\sum_{\substack{ 1\leq r,\abs{k_1},\abs{l_1}\leq \BOUND \\ 1\leq \abs w\leq 2\BOUND^2 \\ \gcd(k_1,l_1)=1 }}
			\dfrac1{r^2\abs w} 
				\sum_{m=1}^{r+1}
				\int_{I_m} \dfrac{\dd\gamma}{f(\gamma)} \lll \\
		\sum_{\substack{ 1\leq r,\abs{k_1},\abs{l_1}\leq \BOUND \\ 1\leq \abs w\leq 2\BOUND^2 \\ \gcd(k_1,l_1)=1 }}
			\dfrac1{r^2\abs w} 
				\dfrac{(r+1)\log r}{\abs{k_1l_1r}} \lll 1,
	\end{gather}
and this completes the proof of Theorem \ref{bam} for $n=3$.\qed

\p {Proof of Theorem \ref{bam} for $n=2$} \label{sectionmain2}
We will briefly sketch how the proof of Theorem \ref{bam} for the case $n=3$ may be modified for the case $n=2$.

By a decomposition of the measure on $\GL_2(\R)/\GL_2(\Z)$ analogous to equation \eqref{decomposition}, it suffices to prove that
\begin{gather}
	\sqrt{\int_{[1,2)^2}\int_{\R^2} \abs{E_{AN}(t)}^2 \psi(a)\dd a\dd \eta} \lll t^{1/2},
\end{gather}
where $\psi(a)\defeq 4\pi\abs{\det A}^2\psi_1(a_1)\psi_2(a_2)$ for the characteristic functions $\psi_1,\psi_2$ of two closed intervals contained in $(0,\infty)$, and where we use the parametrization $N=\mat{1&\eta_1\\0&1}, \eta_1\in[1,2)$, $A=\mat{1/\sqrt{a_1}&0\\0&1/\sqrt{a_2}}, a_i\in (0,\infty)$.

The analog of Claim \ref{vroom} in two dimensions is that it suffices to prove
\begin{gather}
	\int_{[1,2)^2}\int_{\R^2} \abs{E^{\e}_{AN}(t)}^2 \psi(a)\dd a\dd \eta \lll t,
\end{gather}
for $\e\geq 1/t^{1/2}$.

Next, to estimate the behavior of $E_X^\e$, we begin by considering the Fourier transform of the characteristic function $\chi_\Omega$ of the standard unit ball in $\R^2$. It equals (see equation 11 in chapter 6.4 of \cite{SteinShakarchi})
\begin{gather}
	\hat{\chi_{\Omega}}(k)=2\pi\int_0^1 J_0(2\pi\norm k r)r\dd r,
\end{gather}
where we have written $J_\alpha$ for the Bessel function of the first kind of order $\alpha$. Integrating the Taylor series of $J_0$ (see equation 9.1.10 of \cite{handbook}) term by term, we obtain
\begin{gather}
	\hat{\chi_{\Omega}}(k)=\dfrac{J_1(2\pi\norm k)}{\norm k}.
\end{gather}
Using the asymptotics $J_1(x)=\sqrt{\frac2{\pi x}}\cos(x-3\pi/4)+\BigO(x^{-3/2})$ for large $x$ (see equation 9.2.1 of \cite{handbook}), we obtain
\begin{gather}
	\hat{\chi_{\Omega}}(k)=\dfrac{\cos(2\pi\norm k-\frac{3\pi}4)}{\pi\norm k^{3/2}}+\BigO(\norm k^{-5/2}),
\end{gather}
so it follows, as before, that
\begin{gather}
	\hat{\chi_{\Omega_X}}(k)=\abs{\det X}\inv \dfrac{\cos(2\pi\norm k_X-\frac{3\pi}4)}{\pi\norm k_X^{3/2}}+\BigO(\norm k^{-5/2})
\end{gather}
where we have defined $\norm k_X\defeq \|(X\inv)\T k\|$.

Since $E_X^\e(t)=\sum_{k\neq(0,0)}\hat{\chi_{t\Omega_X}}(k)\hat{\rho_\e}(k)=\sum_{k\neq(0,0)}t^2\hat{\chi_{\Omega_X}}(tk)\hat{\rho}(\e k)$, we obtain, as before,
\begin{gather}
	E_X^\e(t)=\abs{\det X}\inv \sum_{k\neq(0,0)}\left(\dfrac{t^2}{t^{3/2}}\dfrac{\cos(2\pi\norm {tk}_X-\frac{3\pi}4)}{\pi\norm k_X^{3/2}}+\dfrac{t^2}{t^{5/2}}\BigO(\norm k^{-5/2})\right)\hat\rho(\e k)\\
	= \abs{\det X}\inv t^{1/2} \left(\sum_{k\neq(0,0)}\dfrac{\cos(2\pi t\norm {k}_X-\frac{3\pi}4)}{\pi\norm k_X^{3/2}}\hat\rho(\e k)\right)+\BigO(1).
\end{gather}
Writing $\cos(x)=(e^{ix}+e^{-ix})/2$ and squaring $E_X^\e$, it follows, analogous to Claim \ref{afterstuff}, since $\hat\rho$ is real-valued, that it suffices to show that
\begin{gather}
	\sum_{k,l\neq(0,0)} \dfrac{\abs{\hat\rho(\e k)\hat\rho(\e l)} }{\|k\|^{3/2}\|l\|^{3/2}} \abs{I_{k,l}(t)} \lll 1, \label{fot}
\end{gather}
	for $\e\geq 1/t^{1/2}$, where
	\begin{gather}
		I_{k,l}(t) \defeq \int_{[1,2)^2}\int_{\R^2} e^{2\pi i t\Phi_{k,l}(AN)}\psi_{k,l}(AN) \dd a\dd \eta, \\
		\Phi_{k,l}(AN)\defeq \pm \|k\|_{AN}\pm \|l\|_{AN}, \\
		\psi_{k,l}(AN)\defeq \left(\dfrac{\|k\|}{\|k\|_{AN}}\right)^{3/2} \left(\dfrac{\|l\|}{\|l\|_{AN}}\right)^{3/2}\psi_1(a_1)\psi_2(a_2),
	\end{gather}
	for all four choices of signs in the definition of $\Phi_{k,l}$.

The rest of the proof consists of bounding different parts of the sum \eqref{fot}. Doing this for $n=2$ amounts to repeating the arguments for $n=3$ with the difference that now $k,l$ are instead in $\Z^2$ and that the exponents of $\norm k$ and $\norm l$ in \eqref{fot} are $3/2$ instead of $2$. Many of the bounds are improved in the case $n=2$; in contrast, most of these fail for $n\geq 4$ if we repeat our method without modification; the technical reason being that the exponents of $\norm k,\norm l$ for $k,l\in \Z^n$ in the analog of \eqref{fot} are $(n+1)/2$, whereas we would need the exponents to be roughly of the order $n$ to get our desired bounds.

\begin{itemize}
	\item We can neglect coordinates larger than $t^{1/2+\delta}$ in magnitude by using the rapid decay of $\hat\rho$, in the same way did it for $n=3$.
	\item We can neglect integer vectors $k,l$ with at least one zero in each vector in the same way we did for $n=3$, since we need only $\int_1^{t^{1/2+\delta}}\dfrac1{r^{3/2}}\dd r\ll 1$.
	\item Assume that $k_1=0, k_2,l_1,l_2\neq 0$. Then, as in the proof of Lemma \ref{k1lemma}, the van der Corput Lemma implies that $\abs{I_{k,l}(t)}\ll \frac1t\frac{\norm l}{\abs{l_1}}.$ Now
		\begin{gather}
			\sum_{\substack{1\leq \abs{k_2},\abs{l_1},\abs{l_2}\leq t^{1/2+\delta} \\ k_1=0 }} \dfrac{1}{\|k\|^{3/2}\|l\|^{3/2}} \frac1t\frac{\norm l}{\abs{l_1}} \leq
			\sum_{1\leq \abs{k_2},\abs{l_1},\abs{l_2}\leq t^{1/2+\delta} } \dfrac1t \dfrac{1}{\abs{k_2}^{3/2}\abs{l_1}\abs{l_2}^{1/2}} \ll \\
			\dfrac1t\cdot  1\cdot\log(t^{1/2+\delta})\cdot (t^{1/2+\delta})^{1/2} \ll 1. \label{blue}
		\end{gather}
	\item Assume that $k_2=0, k_1,l_1,l_2\neq 0$. We follow the proof of Lemma \ref{k2lemma}. The bound \eqref{lasse} still holds for $n=2$ (where we change the integration domain to $\R^2$ instead), so we are left with bounding two sums, one ranging over the condition $\abs{l_2-2\eta_1l_1}\geq 1$, and the other ranging over the condition $\abs{l_2-2\eta_1l_1}< 1$.
	The first sum we treat as follows:
		\begin{gather}
			\int_{[1,2)}\sum_{\substack{1\leq \abs{k_1},\abs{l_1},\abs{l_2}\leq t^{1/2+\delta} \\ k_2 = 0 \\ \abs{l_2-2\eta_1l_1}\geq 1 }} \dfrac1{\norm k^{3/2}\norm l^{3/2}} \abs{ \int_{\R^2} e^{2\pi i \Phi_{k,l}(AN)} \psi_{k,l}(AN)\dd a}\dd{\eta_1} \ll \\
			\sum_{\substack{1\leq \abs{k_1},\abs{l_1},\abs{l_2}\leq t^{1/2+\delta} \\ \abs{l_2-2\eta_1l_1}\geq 1 }} \dfrac1{\abs{k_1}^{3/2}\norm l^{3/2}} \dfrac1t \dfrac{\norm l}{\abs{l_2}} \ll
			\sum_{\substack{1\leq \abs{k_1},\abs{l_1},\abs{l_2}\leq t^{1/2+\delta}  }} \dfrac1t \dfrac1{\abs{k_1}^{3/2}\abs{l_1}^{3/2} \abs{l_2}} \ll 1,
		\end{gather}
		where the last bound is completely analogous to \eqref{blue}. The second sum we treat as follows:
		\begin{gather}
			\int_{[1,2)}\sum_{\substack{1\leq \abs{k_1},\abs{l_1},\abs{l_2}\leq t^{1/2+\delta} \\ k_2 = 0 \\ \abs{l_2-2\eta_1l_1}< 1 }} \dfrac1{\norm k^{3/2}\norm l^{3/2}} \abs{ \int_{\R^2} e^{2\pi i \Phi_{k,l}(AN)} \psi_{k,l}(AN)\dd a}\dd{\eta_1} \ll \\
			\int_{[1,2)}
				\sum_{\substack{1\leq \abs{k_1},\abs{l_1}\leq t^{1/2+\delta}}}
				\sum_{\substack{1\leq \abs{l_2}\leq t^{1/2+\delta}\\ \abs{l_2-2\eta_1l_1}< 1 }}
				\dfrac1{\abs{k_1}^{3/2}\norm l^{3/2}} \dd{\eta_1}.
		\end{gather}
		The condition $\abs{l_2-2\eta_1l_1}< 1$ implies there is at most one value that $l_2$ may assume in the innermost sum where $l_1,\eta_1$ are held fixed, so we may remove the summation over $l_2$, and use the bound $\norm l\geq \abs{l_1}$ for the summand. The sum above is thus bounded by
		\begin{gather}
			\ll  \sum_{\substack{1\leq \abs{k_1},\abs{l_1}\leq t^{1/2+\delta}  }} \dfrac1{\abs{k_1}^{3/2}\abs{l_1}^{3/2}} \ll 1.
		\end{gather}
	\item We need to prove the analog of Lemma \ref{sum1and2}, that is, we need to prove that
		\begin{gather}
			\sum_{\substack{1\leq \abs{k_1},\abs{k_2},\abs{l_1},\abs{l_2}\leq t^{1/2+\delta} \\ \abs{k_1l_2+k_2l_1+2\gamma k_1l_1}<1 }} \dfrac1{\norm k^{3/2}\norm l^{3/2}} \lll 1,
		\end{gather}
	such that the asymptotic constant is independent of $\gamma$, where $\gamma$ is either $-\eta_1\geq 1$ or $0$. But in the proof of Lemma \ref{sum1and2}, we actually prove
		\begin{gather}
			\sum_{\substack{1\leq \abs{k_1},\abs{k_2},\abs{l_1},\abs{l_2}\leq \BOUND \\ \abs{k_1l_2+k_2l_1+2\gamma k_1l_1}<1 }} \dfrac1{\norm k\norm l} \lll 1,
		\end{gather}
	which is a stronger assertion.

	\item Lemma \ref{hessianlemma} still holds for $n=2$ (when integrating instead over $\R^2$). Applying Lemma \ref{hessianlemma} to the sum $\sum_3$ of \eqref{thirdsum}, it now only remains to bound
	\begin{gather}
		\int_{[1,2)} \sum\nolimits_3 \dfrac{1}{\|k\|^{3/2}\|l\|^{3/2}} \abs{\int_{\R^2}e^{2\pi i t\Phi_{k,l}(AN)}\psi_{k,l}(AN)\dd a}\dd \eta \leq \\
		\int_{-2}^{-1}
		\sum_{\substack{ 1\leq\abs {k_1},\abs{k_2},\abs{l_1},\abs{l_2}\leq \BOUND \\ \abs{k_1l_2+k_2l_1+2\gamma k_1 l_1}\geq 1 \\ k_1l_2-k_2l_1\neq 0 }}
		 \dfrac1{\abs{k_1 l_2-k_2 l_1} \abs{k_1l_2+k_2l_1+2\gamma k_1l_1}} \dd\gamma,
	\end{gather}
	but this is precisely the expression \eqref{vom} on page \pageref{vom}, which we have already bounded as part of the proof for $n=2$.
\end{itemize}

This completes the sketch of the proof for $n=2$.\qed

\p {Proof of Theorem \ref{alttheorem} and Corollary \ref{altcorollary}}  \label{proofoversl}
Denote by \[\ESL{f(X)}\defeq\int_{\SL_n(\R)/\SL_n(\Z)} f(X)\dmu{\mu_1}X\] the mean value of $f$ over the set of all lattices with unit determinant, where $\mu_1$ is the normalized Haar measure on $\SL_n(\R)/\SL_n(\Z)$.
We quote the mean value formulas of Siegel and Rogers (see \cite{siegelmean} and Theorem 4 in \cite{rogers}).

\begin{Theorem}[Siegel's mean value formula]
	Suppose that $n\geq 2$. Let $\rho:\R^n\to\R$ be an integrable function, and let $\Lambda\defeq X\Z^n$ for $X\in\SL_n(\R)$. Then
	\[\ESL{\sum_{\substack{u\in \Lambda} }\rho(u)}=\int_{\R^n}\rho(x) \dd x+\rho(0).\]
\end{Theorem}

\begin{Theorem}[Rogers's mean value formula]
	Suppose that $n\geq 3$. Let $\rho:\R^n\times\R^n\to\R$ be a non-negative Borel-measurable function, and let $\Lambda\defeq X\Z^n$ for $X\in\SL_n(\R)$. Then
	\begin{gather}
		\ESL{\sum_{\substack{u,v\in \Lambda } }\rho(u,v)}= 
		\iint_{\R^n\times\R^n}\rho(x,y) \dd x\dd y+\rho(0,0)+ \\
		2\sum_{q=1}^\infty\sum_{\substack{r\geq 1\\ \gcd(q,r)=1}}\dfrac1{q^n}\int_{\R^n}\left(\rho\left(x,\frac qr x\right)+\rho\left(\frac qr x,x\right)\right)\dd x.
	\end{gather}
\end{Theorem}

\begin{proof}[Proof of Theorem \ref{alttheorem}]
Taking $\rho(u)\defeq\chi_{t\Omega}(u)$ in Siegel's mean value formula, we obtain \[\ESL{N_X(t)}=\vol(t\Omega)+1,\] and taking $\rho(u,v)\defeq\chi_{t\Omega}(u)\chi_{t\Omega}(v)$ in Rogers's mean value formula, we obtain
\begin{gather}
	\ESL{N_X(t)^2}=\vol(t\Omega)^2+1+
	4\sum_{q=1}^\infty\sum_{\substack{r\geq 1\\ \gcd(q,r)=1}}\dfrac1{q^n}\int_{\R^n}\chi_{t\Omega}(x)\chi_{t\Omega}\left(\frac qr x\right) \dd x,
\end{gather}
so that
\begin{gather}
	\ESL{N_X(t)^2}-(\vol(t\Omega)^2+1)
	=  4\sum_{\substack{q,r\geq 1\\ \gcd(q,r)=1}}\dfrac1{(qr)^n}\int_{\R^n}\chi_{t\Omega}(qx)\chi_{t\Omega}(rx) \dd x =\\
	4\sum_{\substack{q,r\geq 1\\ \gcd(q,r)=1}}\dfrac1{(qr)^n}\vol\left(\dfrac t{\max(q,r)}\Omega\right)
	=  \sum_{\substack{q,r\geq 1\\ \gcd(q,r)=1}}\dfrac{4\vol(t\Omega)}{(qr)^n \max(q,r)^n} \eqdef c_n\vol(t \Omega),
\end{gather}
where $c_n\geq 4$ is a constant (which is clearly convergent for $n\geq 2$). Thus we have
\begin{gather}
	\ESL{E_X(t)^2} = \ESL{(N_X(t)-\vol(t\Omega))^2} = \\
	\ESL{N_X(t)^2}-2\vol(t\Omega)\ESL{N_X(t)}+\vol(t\Omega)^2 = \\
	c_n\vol(t \Omega)+1-2\vol(t\Omega)=1+(c_n-2)\vol(\Omega)t^n=\Theta(t^n),
\end{gather}
so $\sqrt{\ESL{\abs{E_X(t)}^2}}=\Theta(t^{n/2})$. This completes the proof of Theorem \ref{alttheorem}.
\end{proof}


\begin{proof}[Proof of Corollary \ref{altcorollary}]
	We identify $\GL_n(\R)/\GL_n(\Z)$ with $\GL^+_n(\R)/\SL_n(\Z)$, where $\GL_n^+(\R)$ is the subset of $\GL_n(\R)$ consisting of matrices with positive determinant, and use the decomposition $\GL^+_n(\R)/\SL_n(\Z) = ( \SL_n(\R)/\SL_n(\Z) )\cdot \mathcal D$, where $\mathcal D=\{rI: r>0\}$ is the set of positive multiples of the identity matrix $I$. We identify the Haar measure on $\GL_n^+(\R)/\SL_n(\R)$ with the Haar measure $\mu$ on $\GL_n(\R)$, which is well-known to be bi-invariant. The Haar measure $dr/r$ on $\mathcal D$ is bi-invariant as well since $\mathcal D$ is commutative. Thus the modular functions on these topological groups are identically 1 (see \cite{knapp}). Consequently, Theorem 8.32 from \cite{knapp} implies that
	\begin{gather}
		\int_{\substack{ a \leq \abs{\det X} \leq b}} \abs{E_X(t)}^2 \dmu{\mu}X=
		\int_{\substack{rI\in\mathcal D \\ a \leq r^n \leq b}}
		\int_{\SL_n(\R)/\SL_n(\Z)} \abs{E_{rX}(t)}^2 
		\dmu{\mu_1}X \dfrac{\dd r}{r}.
	\end{gather}
	We have $E_{rX}(t)=E_X(t/r)$ for any $r>0$, so the inner integral can be written as $\ESL{\abs{E_X(t/r)}^2}$. Using the bounds from Theorem \ref{alttheorem} on the inner integral, and bounding the outer integral trivially, we get
	\[
		\int_{L_{a,b}}\abs{E_X(t)}^2\dmu\mu X = \Theta(t^n). \qedhere
	\]
\end{proof}

\section* {Acknowledgements}
I would like to thank my advisor Pär Kurlberg for suggesting this problem to me and for all his help and encouragement.

\bibliographystyle{alpha}
\bibliography{bibs}{}
\end{document}